\documentclass[12pt]{amsart}
\usepackage{geometry}   %????????
\usepackage[colorlinks,citecolor = red, linkcolor=blue,hyperindex]{hyperref}
\usepackage{euscript,eufrak,verbatim, mathrsfs, amscd}
\usepackage[psamsfonts]{amssymb}
\usepackage{bbm}
\usepackage{graphicx}
 \usepackage{float}
\usepackage{float, tikz}

 \usepackage[all, cmtip]{xy}

\usepackage{upref, xcolor, dsfont}
\usepackage{amsfonts,amsmath,amstext,amsbsy, amsopn,amsthm}
\usepackage{enumerate}

\usepackage{url}

\usepackage{mathtools}
\usepackage{bookmark}

 \usepackage{euscript}
\usepackage{helvet}         % selects\textbf{\textbf{•}} Helvetica as sans-serif font
\usepackage{courier}        % selects Courier as typewriter font
\usepackage{type1cm}        % activate if the above 3 fonts are
\usepackage{multicol}        % used for the two-column index
\usepackage[bottom]{footmisc}% places footnotes at page bottom

\newtheorem{theorem}{Theorem}[section]
\newtheorem*{theorem*}{Theorem B}
\newtheorem{lemma}[theorem]{Lemma}

\newtheorem{proposition}[theorem]{Proposition}

\newtheorem*{definition*}{Definition}
\newtheorem*{remark*}{Remark}

\newtheorem*{observation*}{Observation}

\newtheorem*{assumption*}{Assumption}

\theoremstyle{definition}

\theoremstyle{remark}
\newtheorem{remark}{Remark}[section]

\begin{document}

\title[Volterra operators]{Volterra type integration operators between weighted Bergman spaces and Hardy spaces}

\author%[authorlabel1]
{Yongjiang Duan}
\address%[authorlabel1]
{Yongjiang DUAN: School of
Mathematics and Statistics,  Northeast Normal University, Changchun, Jilin 130024, P.R.China}
\email{duanyj086@nenu.edu.cn}
%\email{tba}

%\author%[authorlabel1]
%{Kunyu Guo}
%\address%[authorlabel1]
%{Kunyu GUO: School of Mathematical Sciences, Fudan University, Shanghai 200433, P.R.China}
%\email{kyguo@fudan.edu.cn}

\author%[authorlabel1]
{Siyu Wang}
\address%[authorlabel1]
{Siyu WANG: School of
Mathematics and Statistics, Northeast Normal University, Changchun, Jilin 130024, P.R.China}
\email{wangsy696@nenu.edu.cn}

\author{Zipeng Wang}
\address{Zipeng WANG: College of Mathematics and Statistics, Chongqing University, Chongqing,
401331, P.R.China}
\email{zipengwang@cqu.edu.cn}
\begin{abstract}
Let $\mathcal{D}$ be the class of radial weights on the unit disk which satisfy both forward and reverse doubling conditions. Let $g$ be an analytic function on the unit disk $\mathbb{D}$. We characterize bounded and compact Volterra type integration operators
\[
J_{g}(f)(z)=\int_{0}^{z}f(\lambda)g'(\lambda)d\lambda
\]
between weighted Bergman spaces $L_{a}^{p}(\omega )$ induced by $\mathcal{D}$ weights and Hardy spaces $H^{q}$ for $0<p,q<\infty$.
\end{abstract}

%\subjclass[2010]{Primary 47B35, 47B65, 15B05; Secondary 47D03, 46N30}
%\keywords{Toeplitz operator, Carleson measures, distinct Bergman spaces}

\maketitle
\section{Introduction}
A \textit{weight} $\omega$ in this paper is a positive and integrable function on the unit disk $\mathbb{D}$. For $0<p<\infty$, the \textit{weighted Bergman space $L^p_a(\omega)$} is a linear space of analytic functions on $\mathbb{D}$ satisfying
\[
\|f\|_{L^p(\omega)}:=\Big(\int_{\mathbb{D}}|f(z)|^p\omega(z)dA(z)\Big)^{1/p}<\infty,
\]
where $dA(z)=dxdy/\pi$ is the normalized area measure on $\mathbb{D}$.
If $\omega(z)=\omega_\alpha(z):=(1+\alpha)(1-|z|^2)^\alpha,-1<\alpha<\infty$, then $L^p_a(\omega_\alpha)$ is known as standard weighted Bergman space in literatures. Moreover, $L_a^p(\omega_0)$ is the classical Bergman space, and is denoted by $L_a^p$ \cite{zhu2007}. Recall for $0<p<\infty$, the \textit{Hardy space} consists of analytic functions $f$ on $\mathbb{D}$ such that
\[
\|f\|_{H^p}:=\sup_{0\leq r<1}\Big(\frac{1}{2\pi}\int_0^{2\pi}|f(re^{i\theta})|^pd\theta\Big)^{1/p}<\infty.
\]

Let $g$ be an analytic function on $\mathbb{D}$. We consider the \textit{Volterra type integration operator} $J_{g}$ with symbol $g$
\[
J_{g}(f)(z):=\int_{0}^{z}f(\lambda)g'(\lambda)d\lambda,\quad z\in\mathbb{D},
\]
where $f$ is an analytic function on $\mathbb{D}$.

Pommerenke \cite{Po-1977} proved that $J_g$ is bounded on Hardy space $H^2$ if and only if $g$ is of bounded mean oscillation. The same characterization of the boundedness
of $J_g$ on the general Hardy space $H^p$ was obtained by Aleman and Siskakis \cite{A-S-1995} for $p\geq 1$ and by Aleman and Cima \cite{A-C-2001} for $0<p<1$. In \cite{A-C-2001}, Aleman and Cima further considered any pair of $(p,q)\in (0,\infty)\times(0,\infty)$ and studied the symbol $g$ such that $J_g:H^p\to H^q$ is bounded. Analogous characterization of the symbol $g$ for $J_g$ to be bounded on Bergman spaces can be found in \cite{Siskakis-2006}. Aleman and Siskakis \cite{A-S-1997} established the result on a large class of weighted Bergman spaces, which includes the class of standard weights $\omega_\alpha$ as special examples. Pau and Pel\'{a}ez \cite{Pau-2010} described the $L^{p}_a(\omega)$-$L^{q}_a(\omega)$ boundedness of $J_g$ for $0<p,q<\infty$ when $\omega$ is rapidly decreasing, while Pel\'{a}ez and R\"{a}tty\"{a} \cite{PR} got the corresponding characterizations when $\omega$ is regular or rapidly increasing weight.
Wu \cite{Wu-Scichina} established boundedness criterions of $J_g$ from $L_{a}^p(\omega_\alpha)$ to $H^q$ for all $(p,q)\in (0,\infty)\times (0,\infty)$ but $0<q<\min\{2,p\}<\infty$. Very recently, this gap was filled by Miihkinen, Pau, Per\"{a}l\"{a} and Wang in \cite{Pau-2020} using analysis on tent spaces and some other delicate techniques.
Moreover, by modifying the proofs of the boundedness results, one can get the corresponding compactness results (\cite{Siskakis-2006,Wu-Scichina,Pau-2021}).

The goal of the present paper is to study Volterra type integration operators between weighted Bergman spaces $L_a^p(\omega)$ and Hardy spaces $H^q$ for $0<p,q<\infty$. Here the weight $\omega$  is a $\mathcal{D}$ weight. Recall that the class of \textit{$\mathcal{D}$ weights} \cite{PR19} consists of positive integrable radial functions $\omega$ on $\mathbb{D}$ such that $\omega$ satisfies \textit{both forward and reverse doubling conditions}, i.e., there exist constants $C_1,C_2\in(1,\infty)$ and $C_3\in[1,\infty)$,
\[
C_1\widehat{\omega}\bigg(1-\frac{1-r}{C_2}\bigg)\leq \widehat{\omega}(r)\leq C_3\widehat{\omega}\bigg(\frac{1+r}{2}\bigg), \quad r\in [0,1),
\]
where $\widehat{\omega}(z)=\int_{|z|}^{1}\omega(r)dr$ for all $z\in\mathbb{D}$. The class of $\mathcal{D}$ weights can be regarded as the largest class of radial weights $\omega$ in some sense that most of the results over $L_{a}^{p}(\omega_{\alpha})$ can still hold on $L_{a}^{p}(\omega)$. For instance, the Littlewood-Paley estimate holds for weighted Bergman spaces induced by radial weights if and only if $\omega$ is a $\mathcal{D}$ weight \cite{PR19}.

Let $H(\mathbb{D})$ be the set of analytic functions on $\mathbb{D}$. To describe our results, let us introduce two function spaces, which can be regarded as generalized Bloch spaces (little Bloch spaces). For a radial weight $\omega$ and $\gamma\in\mathbb{R}$, let
\[
\mathcal{B}^{\gamma,\omega}=\Big\{f\in H(\mathbb{D}):
\sup\limits_{z\in\mathbb{D}}\frac{(1-|z|^{2})^{\gamma}}{\omega(z)}|f'(z)|<\infty
\Big\}
\] and
\[
\mathcal{B}_{0}^{\gamma,\omega}=\Big\{f\in H(\mathbb{D}):
\lim\limits_{|z|\to1^{-}}\frac{(1-|z|^{2})^{\gamma}}{\omega(z)}|f'(z)|=0
\Big\}.
\]

Let $\mathbb{T}$ be the unit circle. For an interval $I\subset\mathbb{T}$, we define the \textit{Carleson square}
\[
S(I)=\{z\in\mathbb{D}: z/|z|\in I,~1-|I|\leq |z|<1\},
\]
where $|\cdot|$ is the normalized arc measure on $\mathbb{T}$.
For $0<s<\infty$, a positive measure $\mu$ on $\mathbb{D}$ is a \textit{$s$-Carleson measure
(vanishing $s$-Carleson measure)} \cite[Section 9.2]{zhu2007} if
\[
\sup\limits_{I\subseteq\mathbb{T}}\frac{\mu(S(I))}{|I|^{s}}<\infty\quad
\Big(\lim\limits_{|I|\to 0}\frac{\mu(S(I))}{|I|^{s}}=0\Big).
\]
The (vanishing) $1$-Carleson measure is called (vanishing) Carleson measure.
Let $\xi\in\mathbb{T}$ and $0<\alpha<\infty$, the \textit{Stolz~non-tangential~approach~region} is given by
\[
\Gamma_{\alpha}(\xi)=\bigg\{z\in\mathbb{D}:|z-\xi|<(\alpha+1)(1-|z|)\bigg\}.
\]
For convenience, we denote $\Gamma_{1}(\xi)$ by $\Gamma(\xi)$. Let $0<p<\infty$, $0<q\leq\infty$ and $\nu$ be a positive Borel measure. The \textit{tent space} $T_{q,\nu}^{p}$
consists of $\nu$-equivalence classes of $\nu$-measurable functions $f$ on $\mathbb{D}$ such that
\[
\|f\|_{T_{q,\nu}^{p}}^{p}:=\int_{\mathbb{T}}\bigg(\int_{\Gamma_{\alpha}(\xi)}|f(z)|^{q}d\nu(z)\bigg)^{\frac{p}{q}}d\xi
<\infty,\quad q\in(0,\infty)
\]
and
\[
\|f\|_{T_{q,\nu}^{p}}^{p}:=\int_{\mathbb{T}}\bigg(\textmd{ess}\sup_{z\in\Gamma_{\alpha}(\xi)}|f(z)|\bigg)^{p}d\xi
<\infty,\quad q=\infty,
\]
where the essential supremum is taken with respect to $\nu$.
Note that the tent spaces $T_{q,\nu}^{p}$ does not depend on the choice of $\Gamma_{\alpha}$ \cite{Pau-2020}.
We denote $T_{q,\nu}^{p}$ by $T_{q,\omega,\beta}^{p}$, where $d\nu(z)=\Big(\frac{(1-|z|^{2})^{2}}{\omega(z)}\Big)^{\beta}dA(z)$ and $0<\beta<\infty$.

Our main results are the followings. The first two results characterize the boundedness and compactness of Volterra type integration operator
$J_{g}$ from weighted Bergman spaces to Hardy spaces.

\begin{theorem} \label{b-t1} Let $0<p,q<\infty$, $\omega\in\mathcal{D}$ and $g\in H(\mathbb{D})$. The following statements hold:
\begin{description}
\item[(A)] If $0<p\leq \min\{2,q\}$ or $2<p<q<\infty$, then $J_{g}:L_{a}^{p}(\omega)\to H^{q}$ is bounded if and only if
$
g\in\mathcal{B}^{\gamma,\widehat{\omega}^{\frac{1}{p}}},
$
where $\gamma=1+\frac{1}{q}-\frac{1}{p}$;
\item[(B)] If  $2<p=q<\infty$, then $J_{g}:L_{a}^{p}(\omega)\to H^{q}$ is bounded if and only if the positive measure
$
\mu_{g}
$
is a Carleson measure, where $d\mu_{g}(z)=|g'(z)|^{\frac{2p}{p-2}}\frac{(1-|z|^{2})^{\frac{p+2}{p-2}}}{\widehat{\omega}(z)^{\frac{2}{p-2}}}dA(z)$;
\item[(C)] If $0<q<\infty$ and $p>\max\{2,q\}$, then
$J_{g}:L_{a}^{p}(\omega)\to H^{q}$ is bounded if and only if
$g'\in T_{\frac{2p}{p-2},\widehat{\omega},\frac{2}{p-2}}^{\frac{pq}{p-q}}$;
\item[(D)] If $0<q<p\leq2$, then $J_{g}:L_{a}^{p}(\omega)\to H^{q}$ is bounded if and only if
$B_{g,\omega}\in L^{\frac{pq}{p-q}}(\mathbb{T})$, where
$
B_{g,\omega}(\xi)=\sup\limits_{z\in\Gamma(\xi)}|g'(z)|\frac{(1-|z|^{2})}{\widehat{\omega}(z)^{\frac{1}{p}}},~\xi\in\mathbb{T}.
$
\end{description}
\end{theorem}

\begin{theorem} \label{c-t1} Let $0<p,q<\infty$, $\omega\in\mathcal{D}$ and $g\in H(\mathbb{D})$. The following statements hold:
\begin{description}
\item[(A)] If $0<p\leq \min\{2,q\}$ or $2<p<q<\infty$, then $J_{g}:L_{a}^{p}(\omega)\to H^{q}$ is compact if and only if
$
g\in\mathcal{B}_{0}^{\gamma,\widehat{\omega}^{\frac{1}{p}}},
$
where $\gamma=1+\frac{1}{q}-\frac{1}{p}$;
\item[(B)] If  $2<p=q<\infty$, then $J_{g}:L_{a}^{p}(\omega)\to H^{q}$ is compact if and only if the positive measure
$
\mu_{g}
$
is a vanishing Carleson measure, where $d\mu_{g}(z)=|g'(z)|^{\frac{2p}{p-2}}\frac{(1-|z|^{2})^{\frac{p+2}{p-2}}}{\widehat{\omega}(z)^{\frac{2}{p-2}}}dA(z)$;
\item[(C)] If $0<q<\infty$ and $p>\max\{2,q\}$, then $J_{g}:L_{a}^{p}(\omega)\to H^{q}$ is compact if and only if
$g'\in T_{\frac{2p}{p-2},\widehat{\omega},\frac{2}{p-2}}^{\frac{pq}{p-q}}$;
\item[(D)] If $0<q<p\leq2$, then $J_{g}:L_{a}^{p}(\omega)\to H^{q}$ is compact if and only if
\[
\lim\limits_{r\to1^{-}}\int_{\mathbb{T}}
\sup_{z\in\Gamma(\xi)\backslash r\mathbb{D}}
|g'(z)|^{\frac{pq}{p-q}}\frac{(1-|z|^{2})^{\frac{pq}{p-q}}}
{\widehat{\omega}(z)^{\frac{q}{p-q}}}d\xi=0.
\]
\end{description}
\end{theorem}

The following result gives descriptions for the boundedness and compactness of $J_{g}$ from Hardy spaces to weighted Bergman spaces. For a positive measure $\mu$, define
\[
\widetilde{\mu}(\xi)=\int_{\Gamma(\xi)}\frac{d\mu(\lambda)}{1-|\lambda|^{2}},\quad\xi\in\mathbb{T}.
\]
\begin{theorem} \label{b-t2} Let $0<p,q<\infty$. For $\omega\in\mathcal{D}$ and $g\in H(\mathbb{D})$, let $d\mu_{g}(z)=|g'(z)|^{q}(1-|z|^{2})^{q}\omega(z)dA(z)$. Then
\begin{description}
\item[(A)]  If $0<p\leq q <\infty$, then $J_{g}:H^{p}\to L_{a}^{q}(\omega)$ is bounded (compact) if and only if $\mu_{g}$ is a $q/p$-Carleson measure (vanishing $q/p$-Carleson measure);
%Moreover,
%\[
%\|J_{g}\|_{H^{p}\to L_{a}^{q}(\omega)}^{q}\asymp\sup\limits_{I\subseteq\mathbb{T}}\frac{\mu_{g}(S(I))}{|I|^{\frac{q}{p}}};
%\]
\item[(B)] If $0<q<p<\infty$, then $J_{g}:H^{p}\to L_{a}^{q}(\omega)$ is compact if and only if
$J_{g}:H^{p}\to L_{a}^{q}(\omega)$ is bounded if and only if $\widetilde{\mu_{g}}\in L^{\frac{p}{p-q}}(\mathbb{T})$.
\end{description}
\end{theorem}

%\begin{theorem} \label{c-t2} Let $0<p,q<\infty$. For $\omega\in\mathcal{D}$ and $g\in H(\mathbb{D})$, let %$d\mu_{g}(z)=|g'(z)|^{q}(1-|z|^{2})^{q}\omega(z)dA(z)$. Then
%\begin{description}
%\item[(A)]  If $0<p\leq q <\infty$, then $J_{g}:H^{p}\to L_{a}^{q}(\omega)$ is compact if and only if $\mu_{g}$ is a vanishing $q/p$-Carleson measure.
%Moreover,
%\[
%\|J_{g}\|_{H^{p}\to L_{a}^{q}(\omega)}^{q}\asymp\sup\limits_{I\subseteq\mathbb{T}}\frac{\mu_{g}(S(I))}{|I|^{\frac{q}{p}}};
%\]
%\item[(B)] If $0<q<p<\infty$, then $J_{g}:H^{p}\to L_{a}^{q}(\omega)$ is compact if and only if $\widetilde{\mu}_{g}\in L^{\frac{p}{p-q}}(\mathbb{T})$.
%\end{description}
%\end{theorem}

\begin{remark}\label{r1}
Theorem \ref{b-t1}-Theorem \ref{b-t2} generalize the recent work of \cite{Pau-2020,Pau-2021} from standard weights $\omega_{\alpha}(z)=(1+\alpha)(1-|z|^2)^\alpha dA(z),-1< \alpha<\infty$ to $\mathcal{D}$ weights.
The main strategies are based on a combination of the basic ideas in \cite{Wu-Scichina,Pau-2020,Pau-2021}. Nevertheless, there are two additional obstacles we need to overcome. The first one is the lack of explicit expressions for Bergman kernels of $\mathcal{D}$ weighted Bergman spaces. In addition, the parameter $\alpha$ plays important roles in \cite{Pau-2020,Pau-2021} for standard weighted Bergman spaces.
We use some delicate analysis of the weights to capture the suitable information from $\omega$ instead of $\alpha$.
\end{remark}

%\begin{remark}\label{r2} By an observation of the main results Theorem \ref{b-t1}-Theorem \ref{c-t2} above, it is worth mentioning that for the case when $0<q<p<\infty$, the boundedness and compactness of the Volterra type integration operators are not always equivalent. For the case when $0<q<\infty$ and $p>\max\{2,q\}$, they are equivalent while for the other case of $0<q<p<\infty$, things are different.
%\end{remark}
\begin{remark}
Since the proof of boundedness (Theorem \ref{b-t1}) and compactness (Theorem \ref{c-t1}) of $J_g$ share several overlaps, we just present detailed proof of Theorem \ref{c-t1}.
\end{remark}
Throughout this paper, for a radial weight $\omega$, we denote $\omega(E)=\int_{E}\omega(z)dA(z)$, for any measurable set $E\subseteq\mathbb{D}$. And we will use $a\lesssim b$ to represent that there exists a constant $C=C(\cdot)>0$ satisfying $a\leq Cb$, $a\gtrsim b$ if
there exists a constant $C=C(\cdot)>0$ satisfying $a\geq Cb$, where the constant $C(\cdot)$ depends on the parameters indicated in the parenthesis, varying under different circumstances. Moreover, if $a\lesssim b$ and $b\gtrsim a$, then we denote by $a\asymp b$. Besides, for a number $p>1$, let $p'$ be the conjugate index of $p$ such that $\frac{1}{p}$+$\frac{1}{p'}$=1.
%\blue{\begin{remark}\label{re1} For the case when $0<p<q<\infty$, we give another characterization for the compactness of
%$J_{g}:H^{p}\to L_{a}^{q}(\omega)$ by using little Bloch type condition instead of vanishing Carleson measure type condition.
%Let $0<p<q<\infty$, $\omega\in\mathcal{D}$, $g\in H(\mathbb{D})$ and $d\nu_{g}(z)=|g'(z)|^{q}(1-|z|)^{q}\frac{\widehat{\omega}(z)}{1-|z|}dA(z)$.
%Then $J_{g}:H^{p}\to L_{a}^{q}(\omega)$ is compact
%if and only if
%\[
%\lim\limits_{|z|\to1^{-}}(1-|z|^{2})^{1+\frac{1}{q}-\frac{1}{p}}\widehat{\omega}(z)^{\frac{1}{q}}|g'(z)|=0.
%\]
%\end{remark}}

\section{From Bergman space into Hardy space}
The main goal in this section is to prove Theorem \ref{c-t1}. The whole proof is divided into four parts.

\subsection{Case A} First we are devoted to dealing with the case when $0<p\leq \min\{2,q\}$ or $2<p<q<\infty$.
\begin{proposition} \label{t-c-1} Assume $0<p\leq \min\{2,q\}$ or $2<p<q<\infty$. Let $\omega$ be a $\mathcal{D}$ weight and $\gamma=1+\frac{1}{q}-\frac{1}{p}$. Then the Volterra type integration operator $J_{g}:L_{a}^{p}(\omega)\to H^{q}$ is compact if and only if
$
g\in\mathcal{B}_{0}^{\gamma,\widehat{\omega}^{1/p}},
$
where $\widehat{\omega}(z)=\int_{|z|}^1\omega(r)dr$.
\end{proposition}
To show Proposition \ref{t-c-1}, we first give some useful lemmas.
By a disk version of Calder\'{o}n's theorem, one has the following:
\begin{lemma} \label{le1} \cite[Theorem 1.3]{pav1} Suppose $f$ is analytic on $\mathbb{D}$ with $f(0)=0$. If $0<p<\infty$, then
\[
\|f\|_{H^{p}}^{p}\asymp\int_{\mathbb{T}}\bigg(\int_{\Gamma(\xi)}|f'(z)|^{2}dA(z)\bigg)^{\frac{p}{2}}d\xi.
\]
\end{lemma}
The following Dirichlet-type embedding theorem can be found in \cite{Pau-2020}.
\begin{lemma} \label{le8} Suppose $f$ is analytic on $\mathbb{D}$ with $f(0)=0$. If $0<q\leq2$,  then
\begin{equation}\label{eq19}
\|f\|_{H^{q}}\lesssim\|f'\|_{L_{a}^{q}(\omega_{q-1})}.
\end{equation}
If $0<p<q<\infty$, then
\begin{equation}\label{eq20}
\|f\|_{H^{q}}\lesssim\|f'\|_{L_{a}^{p}(\omega_{p+\frac{p}{q}-2})}.
\end{equation}
\end{lemma}
Now we recall two classes of weights which have close relations with $\mathcal{D}$ weights. A weight $\omega$ is called \textit{regular} \cite{PR} if $\omega$ is continuous and radial on $\mathbb{D}$ such that
$\psi_{\omega}(r)\asymp(1-r)$ for $0\leq r<1$, where $\psi_{\omega}(r)=\frac{\widehat{\omega}(r)}{\omega(r)}$ is the distortion function, which was introduced in \cite{S}. We denote the class of regular weights by $\mathcal{R}$. And a weight $\omega$ is a \textit{$\widehat{\mathcal{D}}$ weight} \cite{PRS} if there exists a positive constant $C$ such that
\[
\widehat{\omega}(r)\leq C \widehat{\omega}(\frac{1+r}{2}),\quad \textmd{for~all}~r\in[0,1).
\]
Note that $\mathcal{R}\subsetneq \mathcal{D}\subsetneq\widehat{\mathcal{D}}$ \cite{PR,PRS,PR19}. Here we list some useful properties of $\widehat{\mathcal{D}}$, $\mathcal{D}$ and regular weights which are important in our subsequent work.
\begin{lemma}\label{reproducing} \cite[Theorem C]{PRS}
Let $1<p<\infty$ and $\omega$ be a $\widehat{\mathcal{D}}$ weight. Then the reproducing kenel of $L_{a}^{2}(\omega)$ at $z$ has the following estimates,
\[
\|K_{z}^{\omega}\|_{L_{a}^{p}(\omega)}^{p}\asymp\int_{|z|}^{1}\frac{dt}{\widehat{\omega}(t)^{p-1}(1-t)^{p}},\quad |z|\to1^{-}.
\]
\end{lemma}
From Lemma \ref{reproducing}, we get
\begin{equation}\label{reproducing-formula}
\|K_{z}^{\omega}\|_{L_{a}^{p}(\omega)}^{p}\asymp\frac{1}{\widehat{\omega}(z)^{p-1}(1-|z|)^{p-1}},\quad z\in\mathbb{D}.
\end{equation}

By \cite[Proposition 5]{PRS} and its proof, we have
\begin{lemma}\label{prop-reverse}
Let $0<p<\infty$ and $\omega$ be a $\mathcal{D}$ weight. Then $\kappa(z)=(1-|z|)^{-1}\widehat{\omega}(z)$ is regular and $\|f\|_{L^{p}(\omega)}\asymp\|f\|_{L^{p}(\kappa)}$, for each analytic function $f$ on $\mathbb{D}$.
\end{lemma}
Denote by $B(z,r)$ and $D(z,r)$ the Euclidean disk and the pseudo-hyperbolic disk centered at $z$ with radius $r$ respectively. The next result shows a nice property of regular weights, which can be found in page 8 of \cite{PR}.
\begin{lemma}\label{regular-lemma} Let $\omega$ be a regular weight.
Then for any $z\in\mathbb{D}$ and $r\in (0,1)$,
\[
\omega(z)\asymp\omega(u),\quad\textmd{for~all}~u\in D(z,r).
\]
\end{lemma}

Combine Lemma \ref{prop-reverse} with Lemma \ref{regular-lemma}, we get the following result.
\begin{lemma}\label{D-regular} Let $\omega\in\mathcal{D}$ and $r\in(0,1)$. Then for any $z\in\mathbb{D}$,
\[
\widehat{\omega}(z)\asymp\widehat{\omega}(u),\quad\textmd{for~all}~u\in D(z,r).
\]
\end{lemma}

Now we are ready to prove Proposition \ref{t-c-1}.

\begin{proof}[Proof of Proposition \ref{t-c-1}] Sufficiency. Let $\{f_{j}\}_{j=1}^{\infty}$ be any sequence in $L_{a}^{p}(\omega)$ such that $\|f_{j}\|_{L_{a}^{p}(\omega)}\leq 1$ and $f_{j}$ uniformly converges to 0 on any compact subset of $\mathbb{D}$. If $g\in\mathcal{B}_{0}^{\gamma,\widehat{\omega}^{1/p}}$, then we will show
\begin{equation*}\label{eq82}
\lim\limits_{j\to\infty}\|J_{g}(f_{j})\|_{H^{q}}=0.
\end{equation*}
For any $\varepsilon>0$, there exists $r\in(0,1)$ such that
for any $z\in\mathbb{D}$ with $r\leq |z|<1$, it holds that
\begin{equation}\label{p1-1}
\frac{(1-|z|^{2})^{\gamma}}{\widehat{\omega}(z)^{1/p}}|g'(z)|<\varepsilon.
\end{equation}
Moreover, we can find a constant $k\in\mathbb{N}$ such that for any $j>k$, $|f_{j}|<\varepsilon$ on $r\mathbb{D}:=\{z:|z|<r\}$.
By Lemma \ref{le8}, for any $j\in\mathbb{N}$, we obtain
\[
\|J_{g}(f_{j})\|_{H^{q}}\lesssim
\|(J_{g}f_{j})'\|_{L_{a}^{p}(\omega_{p\gamma-1})}.
\]
Together with Lemma \ref{prop-reverse} and \eqref{p1-1}, for any $j>k$, we get
\begin{align*}
\|J_{g}(f_{j})\|_{H^{q}}
&\lesssim\bigg(\int_{r\mathbb{D}}|f_{j}(z)|^{p}\frac{\widehat{\omega}(z)}{1-|z|}
|g'(z)|^{p}\frac{(1-|z|^{2})^{p\gamma}}{\widehat{\omega}(z)}dA(z)\bigg)^{\frac{1}{p}}\nonumber\\
&+\bigg(\int_{\mathbb{D}\backslash r\mathbb{D}}|f_{j}(z)|^{p}\frac{\widehat{\omega}(z)}{1-|z|}
|g'(z)|^{p}\frac{(1-|z|^{2})^{p\gamma}}{\widehat{\omega}(z)}dA(z)\bigg)^{\frac{1}{p}}\nonumber\\
&\lesssim \varepsilon+\varepsilon\|f_{j}\|_{L_{a}^{p}(\omega)}\lesssim\varepsilon.
\end{align*}

Necessity. Assume that $J_{g}:L_{a}^{p}(\omega)\to H^{q}$ is compact.
Let $\alpha>\frac{1}{p}$ and $\lambda\in\mathbb{D}$. Take testing functions
\[
f_{\lambda}(z)=\frac{(K_{\lambda}^{\omega}(z))^{\alpha}}{\|K_{\lambda}^{\omega}\|_{L_{a}^{\alpha+1-\frac{1}{p}}(\omega)}^{\alpha+1-\frac{1}{p}}},
\quad z\in\mathbb{D}.
\]
Let $\omega_{2n+1}=\int_{0}^{1}\omega(r)r^{2n+1}dr,~n\in\mathbb{N}\bigcup\{0\}$. It follows from \cite[formula (20)]{PR1} that
\[
|K_{\lambda}^{\omega}(z)|\leq \sum_{n=0}^{\infty}\frac{|z|^{n}}{2\omega_{2n+1}}\asymp\frac{1}{\omega_{1}}+\int_{0}^{|z|}\frac{1}{\widehat{\omega}(t)(1-t)^{2}}dt\lesssim
\frac{1}{\widehat{\omega}(z)(1-|z|)},\quad |z|\to1^{-}.
\]
Together with \eqref{reproducing-formula}, we obtain
\[
|f_{\lambda}(z)|\asymp\widehat{\omega}(\lambda)^{\alpha-\frac{1}{p}}(1-|\lambda|)^{\alpha-\frac{1}{p}}|K_{\lambda}^{\omega}(z)|^{\alpha}
\lesssim\frac{\widehat{\omega}(\lambda)^{\alpha-\frac{1}{p}}(1-|\lambda|)^{\alpha-\frac{1}{p}}}{\widehat{\omega}(z)^{\alpha}(1-|z|)^{\alpha}},
\quad z\in\mathbb{D}.
\]
Then for any compact subset $E\subseteq\mathbb{D}$, we have
\begin{equation*}\label{eqc42}
\lim\limits_{|\lambda|\to1^{-}}\sup\limits_{z\in E}|f_{\lambda}(z)|=0.
\end{equation*}
Moreover, from \eqref{reproducing-formula}, we know
\begin{equation*}\label{eq42}
\|f_{\lambda}\|_{L_{a}^{p}(\omega)}\asymp \widehat{\omega}(\lambda)^{\alpha-\frac{1}{p}}(1-\lambda)^{\alpha-\frac{1}{p}}
\bigg(\int_{\mathbb{D}}|K_{\lambda}^{\omega}(z)|^{\alpha p}\omega(z)dA(z)\bigg)^{\frac{1}{p}}
\asymp1.
\end{equation*}
Therefore, we deduce
\[
\lim\limits_{|\lambda|\to1^{-}}\|J_{g}(f_{\lambda})\|_{H^{q}}=0.
\]
For any $z,\lambda\in\mathbb{D}$,
by using a classical result of the pointwise estimate for $(J_{g}f_{\lambda})'$ (see \cite[Page 85-86]{Ga}), we get
\begin{align}\label{eq41}
|f_{\lambda}(z)g'(z)|
=|(J_{g}f_{\lambda})'(z)|
\lesssim\frac{1}{(1-|z|^{2})^{1+\frac{1}{q}}}\|J_{g}f_{\lambda}\|_{H^{q}}.
\end{align}
By choosing $z=\lambda$ in \eqref{eq41}, together with \eqref{reproducing-formula}, we deduce
\[
\frac{(1-|\lambda|^{2})^{\gamma}}{\widehat{\omega}(\lambda)^{\frac{1}{p}}}|g'(\lambda)|
\lesssim\|J_{g}f_{\lambda}\|_{H^{q}},
\]
from which it follows that
\[
\lim\limits_{|\lambda|\to1^{-}}\frac{(1-|\lambda|^{2})^{\gamma}}{\widehat{\omega}(\lambda)^{\frac{1}{p}}}|g'(\lambda)|=0.
\]
Thus, we complete the proof.
\end{proof}

\subsection{Case B} Now we turn to deal with the case $2<p=q<\infty$.

For $k\in\mathbb{N}$, recall the Rademacher functions
$r_{k}(t):=\textmd{sgn}\{\sin(2^{k}\pi t)\},~t\in[0,1].$ The following Khinchine's inequality and Kahane's inequality \cite{Luecking-1993} are useful in our subsequent work.
\begin{lemma} \label{le-Khinchine} Let $0<p<\infty$. Then for any sequence of complex numbers $\{c_{k}\}_{k=1}^{\infty}$ , it holds that
\[
\bigg(\sum_{k=1}^{\infty}|c_{k}|^{2}\bigg)^{\frac{p}{2}}\asymp\int_{0}^{1}\bigg|\sum_{k=1}^{\infty}c_{k}r_{k}(t)\bigg|^{p}dt.
\]
\end{lemma}
\begin{lemma} \label{le-Kahane} Let $0<p,q<\infty$ and $X$ be a quasi-Banach space with quasi-norm $\|\cdot\|_{X}$. For any sequence $\{x_{k}\}_{k=1}^{\infty}\subseteq X$, it holds that
\[
\bigg(\int_{0}^{1}\big\|\sum_{k=1}^{\infty}r_{k}(t)x_{k}\big\|_{X}^{q}dt\bigg)^{\frac{1}{q}}\asymp
\bigg(\int_{0}^{1}\big\|\sum_{k=1}^{\infty}r_{k}(t)x_{k}\big\|_{X}^{p}dt\bigg)^{\frac{1}{p}}.
\]
\end{lemma}
We first generalize Theorem 1 in \cite{Wu-Scichina} to $\mathcal{D}$ weighted Bergman spaces.
\begin{proposition}\label{vanishing Carleson} Let $\omega\in\mathcal{D}$, $\kappa(z)=\frac{\widehat{\omega}(z)}{1-|z|}$ and $\mu$ be a positive measure on $\mathbb{D}$. Let $p\in(0,+\infty)$, $ s\in[1,+\infty)$ and $\iota\in(0,1)$. Then the following statements are equivalent:

\item[(i)] $\big(\frac{\mu(D(z,\tilde{\delta}))}{\kappa(D(z,\tilde{\delta}))}\big)^{\frac{1}{1-\iota}}\kappa(z)dA(z)$ is a vanishing $s$-Carleson measure, for some $\tilde{\delta}\in(0,1)$;

\item[(ii)]For any $f\in L_{a}^{p}(\omega)$, $|f(z)|^{p\iota}d\mu(z)$ is a vanishing $s(1-\iota)$-Carleson measure, that is, for any $f\in L_{a}^{p}(\omega)$ and any $\varepsilon>0$, there exists $r_{0}\in(0,1)$ such that for any $I\subseteq \mathbb{T}$ with $0<|I|\leq r_{0}$, it holds that
\[
\frac{\int_{S(I)}|f(z)|^{p\iota}d\mu(z)}{|I|^{s(1-\iota)}}\lesssim\varepsilon\|f\|_{L_{a}^{p}(\omega)}^{p\iota},
\]
where the constant of comparison does not depend on the choice of the function $f$.
\end{proposition}
Before approaching the proof of Proposition \ref{vanishing Carleson}, we first give some notations.
For $0<\alpha<\infty$ and $\xi\in\mathbb{T}$, we define the arc on $\mathbb{T}$ as follows,
\[
I_{\alpha}(z)=\bigg\{\xi\in\mathbb{T}:z\in\Gamma_{\alpha}(\xi)\bigg\},\quad z\in\mathbb{D}.
\]
Note that
\begin{equation}\label{eq2}
|I_{\alpha}(z)|\asymp1-|z|,\quad z\in\mathbb{D},
\end{equation}
see \cite[formula (3.3)]{Wu-Scichina}. We denote $I_{\alpha}(z)$ by $I(z)$ if $\alpha=1$.
For $I\subseteq\mathbb{T}$, denote by the \textit{$\alpha$-tent of $I$}
\[
\wedge_{\alpha}(I)=\bigg\{z\in\mathbb{D}:I_{\alpha}(z)\subseteq I\bigg\}.
\]

First observe the relationship between the Carleson square and the $\alpha$-tent of an arc on $\mathbb{T}$.
For any fixed $\alpha>0$, for any $I\subseteq\mathbb{T}$, there exist $I_{1},I_{2}\subseteq\mathbb{T}$ such that
$S(I_{1})\subseteq\wedge_{\alpha}(I)\subseteq S(I_{2})$ and $|I_{1}|\asymp|I|\asymp|I_{2}|$. This makes us to choose $S(I)$ or $\wedge_{\alpha}(I)$ freely
in the definition of vanishing Carleson measure.

%Therefore, for $0<s,\alpha<\infty$,
%the positive measure $\mu$ on $\mathbb{D}$ is a vanishing $s$-Carleson measure if and only if
%\[
%\lim\limits_{|I|\to 0}\frac{\mu(\wedge_{\alpha}(I))}{|I|^{s}}=0.
%\]
\vskip 0.1in
Next we recall some important observations over geometrical regions by Wu in \cite{Wu-Scichina}.
\begin{lemma}\label{o-g} \cite[Lemma 2.2 and Lemma 3.1]{Wu-Scichina} Let $\alpha>0$, $\xi\in\mathbb{T}$ and $I\subseteq\mathbb{T}$.
Let $\tilde{t}=\tanh t$ and $\alpha_{t}=(\alpha+2)e^{2t}-2$ for $t\geq 0$. Then
\[
D(z,\tilde{t})\subseteq\Gamma_{\alpha_{t}}(\xi),\quad \textmd{for~any}~z\in\Gamma_{\alpha}(\xi),
\]
\[
D(z,\tilde{t})\subseteq\wedge_{\alpha}(I), \quad\textmd{for}~z\in\wedge_{\alpha_{t}}(I)
\]
and
\[
D(z,\tilde{t})\subseteq\wedge_{\alpha}(I), \quad\textmd{for}~D(z,\tilde{t})\cap\wedge_{\alpha_{2t}}(I)\neq\emptyset.
\]
\end{lemma}
Let $\{z_{j}\}_{j=1}^{\infty}$ be a sequence in $\mathbb{D}$, $\delta>0$ and $r\in(0,1)$. If $\big|\frac{z_{j}-z_{k}}{1-\bar{z_{j}}z_{k}}\big|\geq\delta$ for any $j\neq k$, then we say the sequence $\{z_{j}\}_{j=1}^{\infty}$ is \textit{$\delta$-separated.} And if $\{z_{j}\}_{j=1}^{\infty}$ satisfies $\mathbb{D}=\bigcup\limits_{j=1}^{\infty}D(z_{j},r)$ and $\frac{r}{5}$-separated, we say $\{z_{j}\}_{j=1}^{\infty}$ is an \textit{$r$-lattice}. Notice that every $z\in\mathbb{D}$ belongs to at most $N=N(r)$ pseudohyperbolic disks $D(z_{j},r)$, where $\{z_{j}\}_{j=1}^{\infty}$ is a separated sequence (see Lemma 12 of Chapter 2 in \cite{Dur} or Lemma 3 in \cite{Luecking-1993}).
\vskip 0.1in
Let $\{z_{k}\}$ be a separated sequence. For any $K\in(1,+\infty)$, we re-index the sequence $\{z_{k}\}$ as follows: For each $j\in\mathbb{N}\bigcup\{0\}$, we denote the points of the sequence $\{z_{k}\}$ in the annulus $A_{j}=\{z\in\mathbb{D}:r_{j}\leq|z|<r_{j+1}\}$ by $\{z_{j,l}\}_{l}$, where $r_{j}=1-K^{-j}$. For $p\in(0,\infty)$, note that a discrete sequence $\{\lambda_{j,l}\}_{j,l=1}^{\infty}\in l^{p}$ if
\[
\bigg\|\{\lambda_{j,l}\}_{j,l=1}^{\infty}\bigg\|_{l^{p}}:=
\bigg(\sum_{j=1}^{\infty}\sum_{l=1}^{\infty}|\lambda_{j,l}|^{p}\bigg)^{\frac{1}{p}}<\infty.
\]
The following (see \cite[Theorem 1]{PRS-2021}) is the atomic decomposition for Bergman spaces induced by $\mathcal{D}$ weights. One can consult \cite{PRS-2021} for more information.
\begin{lemma}\label{atomic decomposition}  Let $0<p<\infty$, $1<K<\infty$, $\omega\in\mathcal{D}$ and $\{z_{k}\}$ be a separated sequence in $\mathbb{D}$. Denote by $\{z_{j,l}\}_{j,l}$ the sequence re-indexed from $\{z_{k}\}$ using the method above. Then there exists a positive constant $M$ ($M$ only depends on the choice of $\omega$ and $p$) such that for any sequence
$\{\lambda_{j,l}\}_{j,l=1}^{\infty}\in l^{p}$, the function
\[
F(z):=\sum_{j,l=1}^{\infty}\lambda_{j,l}\frac{(1-|z_{j,l}|)^{M-\frac{1}{p}}
\widehat{\omega}(z_{j,l})^{-\frac{1}{p}}}{(1-\overline{z_{j,l}}z)^{M}}
\]
is analytic in $\mathbb{D}$ and
\[
\|F\|_{L_{a}^{p}(\omega)}\lesssim \|\{\lambda_{j,l}\}_{j,l=1}^{\infty}\|_{l^{p}}.
\]
\end{lemma}

Now we are in a position to prove Proposition \ref{vanishing Carleson}.
\begin{proof}[Proof of Proposition \ref{vanishing Carleson}] First we prove that $\textmd{(i)}\Rightarrow\textmd{(ii)}$.
Choose $\delta\in(0,\infty)$ such that $\tilde{\delta}:=\tanh\delta$ satisfies the condition in \textmd{(i)}, it follows from Lemma \ref{prop-reverse} that $\kappa(z)=\frac{\widehat{\omega}(z)}{1-|z|}$ is a regular weight. By Lemma \ref{regular-lemma}, we know that
\begin{equation}\label{r-1}
\kappa(z)\asymp\kappa(\zeta),\quad \zeta\in D(z,\tilde{\delta}).
\end{equation}
Let $0<p<\infty$, $1\leq s<\infty$, $0<\iota<1$ and $\alpha>0$. Let $\alpha_{\delta}=(\alpha+2)e^{2\delta}-2$. For any $f\in L_{a}^{p}(\omega)$
and $I\subseteq\mathbb{T}$, using the subharmonicity of $|f|^{p}$, Lemma \ref{o-g} and \eqref{r-1}, we obtain
\begin{align}\label{aeq2}
|f(z)|^{p\iota}
&\lesssim\frac{1}{(1-|z|)^{2}}\int_{B(z,\tilde{\delta}(1-|z|))}|f(\zeta)|^{p\iota}dA(\zeta)\nonumber\\
&\lesssim\frac{1}{\kappa(z)(1-|z|)^{2}}\int_{D(z,\tilde{\delta})}|f(\zeta)|^{p\iota}\kappa(\zeta)dA(\zeta)\nonumber\\
&=\frac{1}{\kappa(z)(1-|z|)^{2}}\int_{\wedge_{\alpha}(I)\cap D(z,\tilde{\delta})}|f(\zeta)|^{p\iota}\kappa(\zeta)dA(\zeta),\quad\textmd{for~any}~ z\in\wedge_{\alpha_{\delta}}(I).
\end{align}
Combining Lemma \ref{prop-reverse}, Fubini's theorem,
$\textmd{H}\ddot{\textmd{o}}\textmd{lder's}$ inequality, \eqref{r-1} with \eqref{aeq2}, for any $I\subseteq\mathbb{T}$, we get
\begin{align}\label{aeq3}
\int_{\wedge_{\alpha_{\delta}}(I)}|f(z)|^{p\iota}d\mu(z)
&\lesssim\int_{\wedge_{\alpha_{\delta}}(I)}\frac{1}{\kappa(z)(1-|z|)^{2}}
\int_{\wedge_{\alpha}(I)\cap D(z,\tilde{\delta})}|f(\zeta)|^{p\iota}\kappa(\zeta)dA(\zeta)d\mu(z)\nonumber\\
&\asymp\int_{\wedge_{\alpha}(I)}\frac{1}{\kappa(\zeta)(1-|\zeta|)^{2}}
\int_{\wedge_{\alpha_{\delta}}(I)\cap D(\zeta,\tilde{\delta})}d\mu(z)|f(\zeta)|^{p\iota}\kappa(\zeta)dA(\zeta)\nonumber\\
&\lesssim\int_{\wedge_{\alpha}(I)}\frac{\mu(D(\zeta,\tilde{\delta}))}{\kappa(D(\zeta,\tilde{\delta}))}|f(\zeta)|^{p\iota}\kappa(\zeta)dA(\zeta)\nonumber\\
%&\leq\bigg(\int_{\wedge_{\sigma}(I)}|f(\zeta)|^{p}\kappa(\zeta)dA(\zeta)\bigg)^{\iota}\cdot
%\bigg(\int_{\wedge_{\sigma}(I)}\bigg(\frac{\mu(D(\zeta,\delta))}{\kappa(D(\zeta,\delta))}\bigg)^{\frac{1}{1-\iota}}
%\kappa(\zeta)dA(\zeta)\bigg)^{1-\iota}\nonumber\\
&\lesssim\|f\|_{L_{a}^{p}(\omega)}^{p\iota}
\bigg(\int_{\wedge_{\alpha}(I)}\bigg(\frac{\mu(D(\zeta,\tilde{\delta}))}{\kappa(D(\zeta,\tilde{\delta}))}\bigg)^{\frac{1}{1-\iota}}
\kappa(\zeta)dA(\zeta)\bigg)^{1-\iota}.
\end{align}
Since $\big(\frac{\mu(D(\zeta,\tilde{\delta}))}{\kappa(D(\zeta,\tilde{\delta}))}\big)^{\frac{1}{1-\iota}}\kappa(\zeta)dA(\zeta)$ is a
vanishing $s$-Carleson measure, then for any $\varepsilon>0$, there exists $r_{1}\in(0,1)$ such that for any $I\subseteq \mathbb{T}$ with $0<|I|\leq r_{1}$, it holds that
\[
\frac{\int_{\wedge_{\alpha}(I)}\big(\frac{\mu(D(\zeta,\tilde{\delta}))}{\kappa(D(\zeta,\tilde{\delta}))}\big)^{\frac{1}{1-\iota}}
\kappa(\zeta)dA(\zeta)}{|I|^{s}}<\varepsilon^{\frac{1}{1-\iota}}.
\]
Combining with \eqref{aeq3}, there exists a constant $C$ ($C$ does not depend on the choice of $f$) such that for any $I\subseteq \mathbb{T}$ with $0<|I|\leq r_{1}$, for any $f\in L_{a}^{p}(\omega)$, we have
\[
\frac{\int_{\wedge_{\alpha_{\delta}}(I)}|f(z)|^{p\iota}d\mu(z)}{|I|^{s(1-\iota)}}
\leq\varepsilon C \|f\|_{L_{a}^{p}(\omega)}^{p\iota},
\]
as desired.
\vskip 0.1in
Next we show $\textmd{(ii)}\Rightarrow\textmd{(i)}$. By the hypothesis, for any $f\in L_{a}^{p}(\omega)$ and $\varepsilon>0$, there exists $r_{2}\in(0,1)$ such that for any $I\subseteq \mathbb{T}$ with $0<|I|\leq r_{2}$, it holds that
\[
\frac{\int_{\wedge_{\alpha}(I)}|f(z)|^{p\iota}d\mu(z)}{|I|^{s(1-\iota)}}\lesssim\varepsilon\|f\|_{L_{a}^{p}(\omega)}^{p\iota}.
\]
Take $\delta>0$ such that $\tilde{\delta}=\tanh \delta\in (0,\frac{1}{2})$. Let $\{z_{k}\}_{k=1}^{\infty}$ be a $\tilde{\delta}$-lattice. Take the constant $M$ and denote by $\{z_{j,l}\}_{j,l=1}^{\infty}$ the sequence re-indexed from $\{z_{k}\}_{k=1}^{\infty}$ as in Lemma \ref{atomic decomposition}.
Hence, by Lemma \ref{atomic decomposition}, for any $\{\lambda_{j,l}\}_{j,l=1}^{\infty}\in l^{p}$, for any $I\subseteq \mathbb{T}$ with $0<|I|\leq r_{2}$, it holds that
\[
\frac{\int_{\wedge_{\alpha}(I)}\big|\sum_{j,l=1}^{\infty}\lambda_{j,l}f_{j,l}(z)\big|^{p\iota}d\mu(z)}{|I|^{s(1-\iota)}}
\lesssim\varepsilon\bigg\|\sum_{j,l=1}^{\infty}\lambda_{j,l}f_{j,l}(z)\bigg\|_{L_{a}^{p}(\omega)}^{p\iota}
\lesssim\varepsilon\|\{\lambda_{j,l}\}_{j,l=1}^{\infty}\|_{l^{p}}^{p\iota},
\]
where
\[
f_{j,l}(z)=\frac{(1-|z_{j,l}|)^{M-\frac{1}{p}}\widehat{\omega}(z_{j,l})^{-\frac{1}{p}}}{(1-\overline{z_{j,l}}z)^{M}}, \quad z\in\mathbb{D}.
\]
For $t\in(0,1)$, replace $\lambda_{j,l}$ by $\lambda_{j,l}r_{j,l}(t)$, where $r_{j,l}(t)$ is the Rademacher function. Then
for any $I\subseteq \mathbb{T}$ with $0<|I|\leq r_{2}$, we have
\[
\frac{\int_{\wedge_{\alpha}(I)}\big|\sum_{j,l=1}^{\infty}\lambda_{j,l}r_{j,l}(t)f_{j,l}(z)\big|^{p\iota}d\mu(z)}{|I|^{s(1-\iota)}}
\lesssim\varepsilon\|\{\lambda_{j,l}\}_{j,l=1}^{\infty}\|_{l^{p}}^{p\iota}.
\]
Integrate both sides of the above inequality over $(0,1)$ with respect to $t$, we obtain
\[
\frac{\int_{\wedge_{\alpha}(I)}\int_{0}^{1}\big|\sum_{j,l=1}^{\infty}\lambda_{j,l}r_{j,l}(t)f_{j,l}(z)\big|^{p\iota}dtd\mu(z)}{|I|^{s(1-\iota)}}
\lesssim\varepsilon\|\{\lambda_{j,l}\}_{j,l=1}^{\infty}\|_{l^{p}}^{p\iota}.
\]
Combining with Lemma \ref{le-Khinchine}, we get
\begin{equation}\label{a-t-eq1}
\int_{\wedge_{\alpha}(I)}\bigg(\sum_{j,l=1}^{\infty}|\lambda_{j,l}|^{2}|f_{j,l}(z)|^{2}\bigg)^{\frac{p\iota}{2}}d\mu(z)
\lesssim\varepsilon|I|^{s(1-\iota)}\|\{\lambda_{j,l}\}_{j,l=1}^{\infty}\|_{l^{p}}^{p\iota},
\end{equation}
for $I\subseteq\mathbb{T}~\textmd{with}~0<|I|\leq r_{2}$.
It follows from \eqref{a-t-eq1} that there exists a constant $N=N(\delta)$ such that
\begin{align}\label{a-t-eq2}
\sum_{\{j,l:D(z_{j,l},2\tilde{\delta})\subseteq\wedge_{\alpha}(I)\}}|\lambda_{j,l}|^{p\iota}
\frac{\mu(D(z_{j,l},2\tilde{\delta}))}{\widehat{\omega}(z_{j,l})^{\iota}(1-|z_{j,l}|)^{\iota}}
&\asymp\sum_{\{j,l:D(z_{j,l},2\tilde{\delta})\subseteq\wedge_{\alpha}(I)\}}
\int_{D(z_{j,l},2\tilde{\delta})}\big(|\lambda_{j,l}|\cdot|f_{j,l}(z)|\big)^{p\iota}d\mu(z)\nonumber\\
&\leq\sum_{\{j,l:D(z_{j,l},2\tilde{\delta})\subseteq\wedge_{\alpha}(I)\}}
\int_{D(z_{j,l},2\tilde{\delta})}\big(\sum_{j,l=1}^{\infty}|\lambda_{j,l}|^{2}\cdot|f_{j,l}(z)|^{2}\big)^{\frac{p\iota}{2}}d\mu(z)\nonumber\\
&\leq N \int_{\wedge_{\alpha}(I)}\big(\sum_{j,l=1}^{\infty}|\lambda_{j,l}|^{2}\cdot|f_{j,l}(z)|^{2}\big)^{\frac{p\iota}{2}}d\mu(z)\nonumber\\
&\lesssim\varepsilon N|I|^{s(1-\iota)}\|\{\lambda_{j,l}\}_{j,l=1}^{\infty}\|_{l^{p}}^{p\iota},
\quad \textmd{for}~I\subseteq\mathbb{T}~\textmd{with}~0<|I|\leq r_{2}.
\end{align}
Since the dual space of $l^{\frac{1}{\iota}}$ can be identified with $l^{\frac{1}{1-\iota}}$, together with \eqref{a-t-eq2}, we deduce
\begin{equation}\label{a-t-eq3}
\bigg[\sum_{\{j,l:D(z_{j,l},2\tilde{\delta})\subseteq\wedge_{\alpha}(I)\}}
\bigg(\frac{\mu(D(z_{j,l},2\tilde{\delta}))}{\widehat{\omega}(z_{j,l})^{\iota}(1-|z_{j,l}|)^{\iota}}\bigg)^{\frac{1}{1-\iota}}\bigg]^{1-\iota}
\lesssim\varepsilon |I|^{s(1-\iota)},
\quad \textmd{for}~I\subseteq\mathbb{T}~\textmd{with}~0<|I|\leq r_{2}.
\end{equation}
Take $t=2\tanh^{-1}(2\tilde{\delta})$. Let $\alpha_{t}=(\alpha+2)e^{2t}-2$. By Lemma \ref{D-regular}, Lemma \ref{o-g} and \eqref{a-t-eq3}, for any $I\subseteq \mathbb{T}$ with
$0<|I|\leq r_{2}$, we get
\begin{align*}\label{a-t-eq4}
\int_{\wedge_{\alpha_{t}}(I)}\bigg(\frac{\mu(D(z,\tilde{\delta}))}{\kappa(D(z,\tilde{\delta}))}\bigg)^{\frac{1}{1-\iota}}\kappa(z)dA(z)
&\leq\sum_{\{j,l:D(z_{j,l},\tilde{\delta})\bigcap\wedge_{\alpha_{t}(I)\neq\varnothing}\}}\int_{D(z_{j,l},\tilde{\delta})}
\bigg(\frac{\mu(D(z,\tilde{\delta}))}{\kappa(D(z,\tilde{\delta}))}\bigg)^{\frac{1}{1-\iota}}\kappa(z)dA(z)\nonumber\\
%&\leq\sum_{j,l:D(z_{j,l},2\delta)\bigcap\wedge_{\sigma_{4\delta}(I)\neq\varnothing}}\int_{D(z_{j,l},\delta)}
%\bigg(\frac{\mu(D(z,\delta))}{\kappa(D(z,\delta))}\bigg)^{\frac{1}{1-\iota}}\kappa(z)dA(z)\nonumber\\
&\leq\sum_{\{j,l:D(z_{j,l},2\tilde{\delta})\subseteq\wedge_{\alpha}(I)\}}\int_{D(z_{j,l},\tilde{\delta})}
\bigg(\frac{\mu(D(z,\tilde{\delta}))}{\kappa(D(z,\tilde{\delta}))}\bigg)^{\frac{1}{1-\iota}}\kappa(z)dA(z)\nonumber\\
%&\lesssim\sum_{j,l:D(z_{j,l},2\delta)\subseteq\wedge_{\sigma}(I)}
%\bigg(\frac{\mu(D(z_{j,l},2\delta))}{\widehat{\omega}(|z_{j,l}|)^{\iota}(1-|z_{j,l}|)^{\iota}}\bigg)^{\frac{1}{1-\iota}}\nonumber\\
&\lesssim\varepsilon^{\frac{1}{1-\iota}}|I|^{s},
\end{align*}
which yields that $\big(\frac{\mu(D(z,\tilde{\delta}))}{\kappa(D(z,\tilde{\delta}))}\big)^{\frac{1}{1-\iota}}\kappa(z)dA(z)$ is a
vanishing $s$-Carleson measure, as desired.
\end{proof}

For a positive measure $\mu$ on $\mathbb{D}$ and $\alpha>0$, define the area integral operator $A_{\mu,\alpha}^{2}$ acting on $f\in H(\mathbb{D})$ by
\[
A_{\mu,\alpha}^{2}(f)(\xi)=\bigg(\int_{\Gamma_{\alpha}(\xi)}|f(z)|^{2}\frac{d\mu(z)}{1-|z|}\bigg)^{\frac{1}{2}},
\quad\xi\in\mathbb{T}.
\]
For a $\mathcal{D}$ weight $\omega$, $0<q,\alpha<\infty$ and $g\in H(\mathbb{D})$, it follows from Lemma \ref{le1} that the Volterra type integration operator is intimately related to area integral operator as follows,
\begin{equation}\label{area integral}
\|J_{g}(f)\|_{H^{q}}\asymp\|A_{\mu_{g},\alpha}^{2}(f)\|_{L^{q}(\mathbb{T})},\quad \textmd{for~any}~f\in L_{a}^{p}(\omega),
\end{equation}
where $d\mu_{g}(z)=|g'(z)|^{2}(1-|z|)dA(z)$.
\vskip 0.1in
\begin{proposition} \label{t-c-2} Let $2<p<\infty$, $\omega\in\mathcal{D}$, $g\in H(\mathbb{D})$ and
\[
d\mu_{g}(z)=|g'(z)|^{\frac{2p}{p-2}}\frac{(1-|z|^{2})^{\frac{p+2}{p-2}}}{\widehat{\omega}(z)^{\frac{2}{p-2}}}dA(z).
\]
Then the Volterra type
integration operator $J_{g}:L_{a}^{p}(\omega)\to H^{p}$ is compact if and only if $\mu_{g}$ is a vanishing Carleson measure.
\end{proposition}

\begin{proof} Sufficiency.
Let $\{f_{j}\}_{j=1}^{\infty}$ be any sequence in $L_{a}^{p}(\omega)$ with $\|f_{j}\|_{L_{a}^{p}(\omega)}\leq 1$ and $f_{j}$ converges to 0 uniformly on any compact subset of $\mathbb{D}$.
Since $\mu_{g}$ is a vanishing Carleson measure, for any $\varepsilon>0$, there exists $r_{0}\in(0,1)$ such that for any $I\subseteq\mathbb{T}$ with $|I|\leq 1-r_{0}$, it holds that
\begin{equation}\label{eq87}
\frac{\mu_{g}(S(I))}{|I|}<\varepsilon^{\frac{2p}{p-2}}.
\end{equation}
Moreover, there also exists $k_{0}\in\mathbb{N}$ such that $j>k_{0}$, $|f_{j}|<\varepsilon$ on $r_{0}\mathbb{D}$.
Let
\[
d\mu_{f_{j},g}(z)=|f_{j}(z)|^{2}|g'(z)|^{2}(1-|z|^{2})dA(z).
\]
By virtue of \eqref{eq87}, we obtain
\begin{equation}\label{n-a-1}
\frac{\mu_{g}(S(I)\backslash r_{0}\mathbb{D})}{|I|}<\varepsilon^{\frac{2p}{p-2}},\quad \textmd{for~any}~I\subseteq\mathbb{T} ~\textmd{with}~|I|\leq 1-r_{0}.
\end{equation}
In fact, we have
\begin{equation}\label{eq97}
\frac{\mu_{g}(S(I)\backslash r_{0}\mathbb{D})}{|I|}<2\varepsilon^{\frac{2p}{p-2}},\quad \textmd{for~any}~I\subseteq\mathbb{T}.
\end{equation}
For any $|I|>1-r_{0}$, there exists $k\in\mathbb{N}$ such that $(k-1)(1-r_{0})<|I|\leq k(1-r_{0})$. There also exist intervals $I_{l}$ on $\mathbb{T}$, $l=1,2,\cdots,k$ such that $|I_{l}|=1-r_{0}$, $l=1,2,\cdots,k$ and $I\subseteq\bigcup\limits_{l=1}^{k}I_{l}$. Notice that
$k(1-r_{0})<2|I|$. Combining with \eqref{n-a-1}, we deduce
\[
\mu_{g}(S(I)\backslash r_{0}\mathbb{D})
\leq\sum_{l=1}^{k}\mu_{g}(S(I_{l})\backslash r_{0}\mathbb{D})
<\varepsilon^{\frac{2p}{p-2}}k(1-r_{0})<\varepsilon^{\frac{2p}{p-2}}2|I|,
\]
which yields that \eqref{eq97} holds.

Let
\[
\kappa(z)=\frac{\widehat{\omega}(z)}{1-|z|},\quad z\in\mathbb{D}.
\]
Since $\chi_{\mathbb{D}\backslash r_{0}\mathbb{D}}\mu_{g}$ is a Carleson measure, by \cite[Theorem 9.4]{D}, Lemma \ref{prop-reverse}, $\textmd{H}\ddot{\textmd{o}}\textmd{lder's}$ inequality and \eqref{eq97}, for any $h\in H^{\frac{p}{p-2}}$, we have
\begin{align}\label{eq85}
\int_{\mathbb{D}\backslash r_{0}\mathbb{D}}|h(z)|d\mu_{f_{j},g}(z)&\leq
\bigg(\int_{\mathbb{D}\backslash r_{0}\mathbb{D}}|h(z)|^{\frac{p}{p-2}}d\mu_{g}(z)\bigg)^{\frac{p-2}{p}}\|f_{j}\|_{L_{a}^{p}(\kappa)}^{2}\nonumber\\
&\lesssim\|f_{j}\|_{L_{a}^{p}(\omega)}^{2}\|h\|_{H^{\frac{p}{p-2}}}\cdot
\bigg(\sup\limits_{ I\subseteq\mathbb{T}}
\frac{\mu_{g}(S(I)\backslash r_{0}\mathbb{D})}{|I|}\bigg)^{\frac{p-2}{p}}\nonumber\\
&\lesssim\varepsilon^{2}\|h\|_{H^{\frac{p}{p-2}}}.
\end{align}
Furthermore,
it follows from Lemma \ref{le1} that
\begin{equation}\label{eq86}
\|J_{g}(f_{j})\|_{H^{p}}^{p}\asymp\int_{\mathbb{T}}\bigg(\int_{\Gamma(\xi)}|f_{j}(z)|^{2}|g'(z)|^{2}dA(z)\bigg)^{\frac{p}{2}}d\xi.
\end{equation}
Recall
\begin{equation}\label{eq84}
\widetilde{\mu_{f_{j},g}}(\xi)=\int_{\Gamma(\xi)}\frac{d\mu_{f_{j},g}(z)}{1-|z|^{2}},\quad\xi\in\mathbb{T}.
\end{equation}
%\begin{equation}\label{eq89}
%\widetilde{\mu}_{f_{j},g}^{r^{1}_{0}}(\xi)=\int_{\Gamma(\xi)}\frac{d\mu_{f_{j},g}^{r^{1}_{0}}(z)}{1-|z|^{2}}~\textmd{and}~
%\widetilde{\mu}_{f_{j},g}^{r^{2}_{0}}(\xi)=\int_{\Gamma(\xi)}\frac{d\mu_{f_{j},g}^{r^{2}_{0}}(z)}{1-|z|^{2}},
%\quad \xi\in\mathbb{T}.
%\end{equation}
By \cite[Theorem E]{Pau-2016}, for any $j\in\mathbb{N}$, we have
\[
\|Id\|_{H^{\frac{p}{p-2}}\to L^{1}(\mu_{f_{j},g}
\chi_{\mathbb{D}\backslash r\mathbb{D}})}\asymp\bigg[\int_{\mathbb{T}}\bigg(\int_{\Gamma(\xi)\backslash r_{0}\mathbb{D}}\frac{d\mu_{f_{j},g}(z)}{1-|z|^{2}}\bigg)^{\frac{p}{2}}d\xi\bigg]^{\frac{2}{p}}.
\]
%\[
%\|\widetilde{\mu}_{f_{j},g}^{r^{2}_{o}}\|_{L^{\frac{p}{2}}(\mathbb{T})}\asymp\|Id\|_{H^{\frac{p}{p-2}}\to %L^{1}(\mu_{f_{j},g}^{r^{2}_{0}})}.
%\]
Together with \eqref{eq85}, \eqref{eq86} and \eqref{eq84}, for $j>k_{0}$, we deduce
\[
\|J_{g}(f_{j})\|_{H^{p}}\asymp\|\widetilde{\mu_{f_{j},g}}\|_{L^{\frac{p}{2}}(\mathbb{T})}^{\frac{1}{2}}
\lesssim\varepsilon.
\]
Hence, we know that $J_{g}$ is a compact operator as desired.
\vskip 0.1in
Necessity. Assume that $J_{g}:L_{a}^{p}(\omega)\to H^{p}$ is compact. Take a positive constant $\alpha$. Let $d\nu_{g}(z)=|g'(z)|^{2}(1-|z|)dA(z)$ and $A_{\nu_{g},\alpha}^{2}$ be the area integral operator from $L_{a}^{p}(\omega)$ to $L^{p}(\mathbb{T})$.
Since the dual space of $L^{\frac{p}{2}}(\mathbb{T})$ can be identified with $L^{\frac{p}{p-2}}(\mathbb{T})$,
together with Fubini's theorem, for any $h\in L^{\frac{p}{p-2}}(\mathbb{T})$ with $h\geq0$ and
$f\in L_{a}^{p}(\omega)$, we obtain
\begin{align}\label{aeq4}
\int_{\mathbb{T}}h(\xi)\big(A_{\nu_{g},\alpha}^{2}(f)(\xi)\big)^{2}d\xi%&=
%\int_{\mathbb{T}}g(\xi)\int_{\Gamma_{\sigma}(\xi)}|f(z)|^{2}\frac{1}{1-|z|}d\nu_{g}(z)d\xi\nonumber\\
=\int_{\mathbb{D}}\frac{1}{1-|z|}\int_{I_{\alpha}(z)}h(\xi)d\xi|f(z)|^{2}d\nu_{g}(z).
\end{align}
Notice that for any $I\subseteq\mathbb{T}$, there exists $a\in\mathbb{D}$ such that $I=I_{\alpha}(a)$.
For any $a\in\mathbb{D}$, take testing functions
\[
h_{a}(\xi)=\chi_{I_{\alpha}(a)}(\xi),\quad\xi\in\mathbb{T}.
\]
Notice that
\begin{equation}\label{aeq5}
\|h_{a}\|_{L^{\frac{p}{p-2}}(\mathbb{T})}\asymp(1-|a|)^{\frac{p-2}{p}}
\end{equation}
and
\begin{equation}\label{aeq6}
\int_{I_{\alpha}(z)}h_{a}(\xi)d\xi\asymp1-|z|,\quad z\in\wedge_{\alpha}(I_{\alpha}(a)).
\end{equation}
Combining \eqref{aeq4}, \eqref{aeq5}, \eqref{aeq6} with $\textmd{H}\ddot{\textmd{o}}\textmd{lder's}$ inequality, for any $a\in\mathbb{D}$, we deduce
\begin{align*}
\int_{\wedge_{\alpha}(I_{\alpha}(a))}|f(z)|^{2}d\nu_{g}(z)
&\lesssim\int_{\mathbb{D}}\frac{1}{1-|z|}\int_{I_{\alpha}(z)}h_{a}(\xi)d\xi|f(z)|^{2}d\nu_{g}(z)\nonumber\\
&=\int_{\mathbb{T}}h_{a}(\xi)\big(A_{\nu_{g},\alpha}^{2}(f)(\xi)\big)^{2}d\xi\nonumber\\
&\leq \|h_{a}\|_{L^{\frac{p}{p-2}}(\mathbb{T})}\|A_{\nu_{g},\alpha}^{2}(f)\|_{L^{p}(\mathbb{T})}^{2}\nonumber\\
&\asymp(1-|a|)^{\frac{p-2}{p}}\|A_{\nu_{g},\alpha}^{2}(f)\|_{L^{p}(\mathbb{T})}^{2},
\end{align*}
from which it follows that
\begin{equation}\label{aeq7}
\frac{\int_{\wedge_{\alpha}(I_{\alpha}(a))}|f(z)|^{2}d\nu_{g}(z)}{(1-|a|)^{\frac{p-2}{p}}}
\lesssim\|A_{\nu_{g},\alpha}^{2}(f)\|_{L^{p}(\mathbb{T})}^{2}.
\end{equation}

For any $\zeta\in\mathbb{D}$ and $\beta>0$, consider test functions
\[
f_{\zeta,\beta}(z)
:=\frac{(1-|\zeta|)^{\frac{\beta+1}{p}}}{(1-\bar{\zeta}z)^{\frac{\beta+1}{p}}\widehat{\omega}(\zeta)^{\frac{1}{p}}(1-|\zeta|)^{\frac{1}{p}}}
,\quad z\in\mathbb{D}.
\]
Recall Lemma A(vi) in \cite{PRS} states that there exists $\gamma_{0}=\gamma_{0}(\omega)>0$ such that for any $\gamma\geq \gamma_{0}$, it holds that
\[
\int_{\mathbb{D}}\frac{\omega(z)}{|1-\bar{\zeta}z|^{\gamma+1}}dA(z)\asymp\frac{\widehat{\omega}(\zeta)}{(1-|\zeta|)^{\gamma}},
\quad \zeta\in\mathbb{D}.
\]
Hence, we can choose $\gamma$ large enough such that $\|f_{\zeta,\gamma}\|_{L_{a}^{p}(\omega)}\asymp1$ and $f_{\zeta,\gamma}$ converges uniformly on any compact subset of $\mathbb{D}$ as $|\zeta|\to 1^{-}$. Combining with \eqref{aeq7}, we obtain that for any $\zeta\in\mathbb{D}$ and $a\in\mathbb{D}$, it holds that
\begin{align}\label{aaeq1}
\frac{\int_{\wedge_{\alpha}(I_{\alpha}(a))}|f_{\zeta,\gamma}(z)|^{2}d\nu_{g}(z)}{(1-|a|)^{\frac{p-2}{p}}}
\lesssim\|A_{\nu_{g},\alpha}^{2}(f_{\zeta,\gamma})\|_{L^{p}(\mathbb{T})}^{2}.
\end{align}
In view of the compactness of $J_{g}$ and \eqref{area integral}, for any $\varepsilon>0$, there exists $r_{1}\in(0,1)$ such that for any $\zeta\in\mathbb{D}$ with $r_{1}\leq|\zeta|<1$, it holds that $\|A_{\nu_{g},\alpha}^{2}(f_{\zeta,\gamma})\|_{L^{p}(\mathbb{T})}^{2}<\varepsilon$.
Notice that there exists $r_{2}\in(r_{1},1)$ such that for any $a\in\mathbb{D}$ with $r_{2}\leq|a|<1$,
\[
D(z,r)\subseteq \mathbb{D}\backslash r_{1}\mathbb{D}, \quad\textmd{for~any}~z\in\wedge_{\alpha}(I_{\alpha}(a)).
\]
For $a\in\mathbb{D}$ with $r_{2}\leq|a|<1$, take $\zeta\in\mathbb{D}$ with
\[
\zeta\in \bigcup_{z\in\wedge_{\alpha}(I_{\alpha}(a))}D(z,r).
\]
Together with \eqref{aaeq1}, for any $a\in\mathbb{D}$ with $r_{2}\leq|a|<1$, we deduce
\begin{equation}\label{aaeq2}
\frac{\int_{\wedge_{\alpha}(I_{\alpha}(a))}\frac{1}{\widehat{\omega}(z)^{\frac{2}{p}}(1-|z|)^{\frac{2}{p}}}d\nu_{g}(z)}{(1-|a|)^{\frac{p-2}{p}}}
\lesssim\varepsilon.
\end{equation}
Next we show that for any $f\in L_{a}^{p}(\omega)$, $|f(z)|^{2}d\nu_{g}(z)$ is a vanishing $\frac{p-2}{p}$-Carleson measure.
By the subharmonicity of $|f|^{p}$, Lemma \ref{prop-reverse}, Lemma \ref{D-regular} and \eqref{aaeq2}, for any $a\in\mathbb{D}$ with $r_{2}\leq|a|<1$, we obtain
\begin{align}\label{aeq8}
\frac{\int_{\wedge_{\alpha}(I_{\alpha}(a))}|f(z)|^{2}d\nu_{g}(z)}{(1-|a|)^{\frac{p-2}{p}}}
&\lesssim\frac{\int_{\wedge_{\alpha}(I_{\alpha}(a))}
\bigg(\frac{1}{\widehat{\omega}(z)(1-|z|)}\int_{D(z,r)}|f(u)|^{p}\frac{\widehat{\omega}(u)}{1-|u|}dA(u)\bigg)^{\frac{2}{p}}d\nu_{g}(z)}
{(1-|a|)^{\frac{p-2}{p}}}\nonumber\\
&\leq
\|f\|_{L_{a}^{p}(\omega)}^{2}\frac{\int_{\wedge_{\alpha}(I_{\alpha}(a))}\frac{1}{\widehat{\omega}(z)^{\frac{2}{p}}(1-|z|)^{\frac{2}{p}}}d\nu_{g}(z)}
{(1-|a|)^{\frac{p-2}{p}}}\nonumber\\
&\lesssim\varepsilon\|f\|_{L_{a}^{p}(\omega)}^{2}.
\end{align}
It follows from Proposition \ref{vanishing Carleson} and \eqref{aeq8} that there exits $t\in(0,1)$ such that
\[
\bigg(\frac{\nu_{g}(D(z,t))}{\widehat{\omega}(z)(1-|z|)}\bigg)^{\frac{p}{p-2}}\frac{\widehat{\omega}(z)}{1-|z|}dA(z)
\]
is a vanishing Carleson measure.
By the subharmonicity of $|g'|^{2}$, we deduce
\[
|g'(z)|^{2}\frac{(1-|z|)^{2}}{\widehat{\omega}(z)}\lesssim\frac{1}{\widehat{\omega}(z)(1-|z|)}
\int_{D(z,t)}|g'(\zeta)|^{2}(1-|\zeta|)dA(\zeta)=\frac{\nu_{g}(D(z,t))}{\widehat{\omega}(z)(1-|z|)},
\]
which yields that
\[
\bigg(|g'(z)|^{2}\frac{(1-|z|)^{2}}{\widehat{\omega}(z)}\bigg)^{\frac{p}{p-2}}\frac{\widehat{\omega}(z)}{1-|z|}dA(z)
\]
is a vanishing Carleson measure.
Thus, we conclude that $\mu_{g}$ is a vanishing Carleson measure.
\end{proof}

\subsection{Case C}
Now we are devoted to dealing with the case $0<q<\infty$ and $p>\max\{2,q\}$.

For a separated sequence $\{a_{k}\}_{k=1}^{\infty}\subset\mathbb{D}$, let
$\nu=\sum\limits_{k=1}^{\infty}\delta_{a_{k}}$ where $\delta_{a_{k}}$ is
Dirac measure at point $a_{k}$, then we denote $T_{q,\nu}^{p}$ by $T_{q}^{p}(\{a_{k}\})$.
For $0<p<\infty$ and $0<q\leq\infty$, if
\[
\|\{b_{k}\}\|_{T_{q}^{p}(\{a_{k}\})}^{p}:=\int_{\mathbb{T}}\bigg(\sum_{a_{k}\in\Gamma(\xi)}|b_{k}|^{q}\bigg)^{\frac{p}{q}}d\xi<\infty,
\quad q\in(0,\infty)
\]
then we say the sequence $\{b_{k}\}_{k=1}^{\infty}\in T_{q}^{p}(\{a_{k}\})$, and if
\[
\|\{b_{k}\}\|_{T_{q}^{p}(\{a_{k}\})}^{p}:=\int_{\mathbb{T}}\bigg(\sup_{a_{k}\in\Gamma(\xi)}|b_{k}|\bigg)^{p}d\xi<\infty,
\quad q=\infty.
\]
we say the sequence $\{b_{k}\}_{k=1}^{\infty}\in T_{\infty}^{p}(\{a_{k}\})$.
\vskip 0.1in
The following two results characterize the dual space and the factorization of tent spaces of sequences respectively (see \cite{Milos} and \cite{Pau-2020}).
\begin{lemma} \label{le2} Let $1<p<\infty$ and $\{a_{k}\}_{k=1}^{\infty}$ be a separated sequence.
\begin{itemize}
\item If $1<q<\infty$, then
the dual space of $T_{q}^{p}(\{a_{k}\})$ can be identified with $T_{q'}^{p'}(\{a_{k}\})$ via the pairing,
\[
\langle\{b_{k}\},\{d_{k}\}\rangle_{T_{2}^{2}(\{a_{k}\})}=\sum\limits_{k=1}^{\infty}b_{k}\bar{d_{k}}(1-|a_{k}|^{2}),
\quad \{b_{k}\}\in T_{q}^{p}(\{a_{k}\}),~\{d_{k}\}\in T_{q'}^{p'}(\{a_{k}\}).
\]
\item If $0<q\leq1$, the the dual space of $T_{q}^{p}(\{a_{k}\})$ can be identified with $T_{\infty}^{p'}(\{a_{k}\})$ via the pairing,
\[
\langle\{b_{k}\},\{d_{k}\}\rangle_{T_{2}^{2}(\{a_{k}\})}=\sum\limits_{k=1}^{\infty}b_{k}\bar{d_{k}}(1-|a_{k}|^{2}),
\quad \{b_{k}\}\in T_{q}^{p}(\{a_{k}\}),~\{d_{k}\}\in T_{\infty}^{p'}(\{a_{k}\}).
\]
\end{itemize}
\end{lemma}
\begin{lemma} \label{le4}  Let $0<p,q<\infty$ and $\{a_{k}\}_{k=1}^{\infty}$ be an $r$-lattice. If $p<p_{1},p_{2}<\infty$, $q<q_{1},q_{2}<\infty$, $\frac{1}{p_{1}}+\frac{1}{p_{2}}=\frac{1}{p}$ and $\frac{1}{q_{1}}+\frac{1}{q_{2}}=\frac{1}{q}$.
Then
\[
T_{q}^{p}(\{a_{k}\})=T_{q_{1}}^{p_{1}}(\{a_{k}\})\cdot T_{q_{2}}^{p_{2}}(\{a_{k}\}),
\]
where $T_{q_{1}}^{p_{1}}(\{a_{k}\})\cdot T_{q_{2}}^{p_{2}}(\{a_{k}\})=\{\{b_{k}\}:~b_{k}=c_{k}d_{k},~\{c_{k}\}\in T_{q_{1}}^{p_{1}}(\{a_{k}\}),\{d_{k}\}\in T_{q_{2}}^{p_{2}}(\{a_{k}\})\}$. Moreover, for any
$\{b_{k}\}\in T_{q}^{p}(\{a_{k}\})$, it holds that
\[
\inf\|\{c_{k}\}\|_{T_{q_{1}}^{p_{1}}(\{a_{k}\})}\cdot
\|\{d_{k}\}\|_{T_{q_{2}}^{p_{2}}(\{a_{k}\})}\lesssim \|\{b_{k}\}\|_{T_{q}^{p}(\{a_{k}\})},
\]
where the infimum above is taken over all possible factorizations of $\{b_{k}\}$.
\end{lemma}

\vskip 0.1in
We also need the following useful estimates.
\vskip 0.1in
\begin{lemma} \label{le3} \cite[Lemma 4]{Milos} Let $0<s<\infty$, $\iota>\max\{1,\frac{1}{s}\}$ and $\mu$ be a positive measure on $\mathbb{D}$. Then
\[
\int_{\mathbb{T}}\bigg(\int_{\mathbb{D}}\bigg(\frac{1-|z|^{2}}{|1-z\bar{\xi}|}\bigg)^{\iota}d\mu(z)\bigg)^{s}d\xi
\asymp\int_{\mathbb{T}}\mu(\Gamma(\xi))^{s}d\xi.
\]
\end{lemma}

Using similar analysis to that in \cite[Proposition 14]{PRS}, we get the following result.
\begin{proposition} \label{le5} Let $1<p<\infty$ and $\omega\in\mathcal{D}$. Assume that $\{a_{k}\}_{k=1}^{\infty}\subseteq \mathbb{D}\backslash\{0\}$
is a separated sequence. Then
\[
F:=\sum_{k=1}^{\infty}c_{k}\frac{K_{a_{k}}^{\omega}}{\|K_{a_{k}}^{\omega}\|_{L_{a}^{2-\frac{1}{p}}(\omega)}^{2-\frac{1}{p}}}
\in H(\mathbb{D}).
\]
Moreover, $\|F\|_{L_{a}^{p}(\omega)}\lesssim\|\{c_{k}\}_{k=1}^{\infty}\|_{l^{p}}$ for all $\{c_{k}\}_{k=1}^{\infty}\in l^{p}$.
\end{proposition}
\begin{proof} Let $\kappa(z)=\frac{\widehat{\omega}(z)}{1-|z|},~z\in\mathbb{D}$ and $\{c_{k}\}_{k=1}^{\infty}\in l^{p}$. From Lemma \ref{prop-reverse}, we know that $\kappa$ is regular. Assume that $\rho\in(0,1)$. For any $z\in\mathbb{D}$ with $0<|z|\leq\rho$, it follows from
Lemma \ref{regular-lemma}, $\textmd{H}\ddot{\textmd{o}}\textmd{lder's}$ inequality and \eqref{reproducing-formula} that
\begin{align*}
|F(z)|
%=\bigg|\sum_{k=1}^{\infty}c_{k}\frac{K_{a_{k}}^{\omega}(z)}{\|K_{a_{k}}^{\omega}\|_{L_{a}^{2-\frac{1}{p}}(\omega)}^{2-\frac{1}{p}}}\bigg|
\leq&\bigg(\sum_{k=1}^{\infty}|c_{k}|^{p}\bigg)^{\frac{1}{p}}
\bigg[\sum_{k=1}^{\infty}\bigg(\frac{|K_{a_{k}}^{\omega}(z)|}
{\|K_{a_{k}}^{\omega}\|_{L_{a}^{2-\frac{1}{p}}(\omega)}^{2-\frac{1}{p}}}\bigg)^{p'}\bigg]^{\frac{1}{p'}}\nonumber\\
\lesssim&\bigg(\sum_{k=1}^{\infty}|c_{k}|^{p}\bigg)^{\frac{1}{p}}\sup\limits_{k\geq1}|K_{a_{k}}^{\omega}(z)|
\bigg(\sum_{k=1}^{\infty}\widehat{\omega}(a_{k})(1-|a_{k}|)\bigg)^{\frac{1}{p'}}\nonumber\\
%=&\bigg(\sum_{k=1}^{\infty}|c_{k}|^{p}\bigg)^{\frac{1}{p}}\sup\limits_{k\geq1}|K_{a_{k}}^{\omega}(z)|
%\bigg(\sum_{k=1}^{\infty}\kappa(a_{k})(1-|a_{k}|)^{2}\bigg)^{\frac{1}{p'}}\nonumber\\
\asymp&\bigg(\sum_{k=1}^{\infty}|c_{k}|^{p}\bigg)^{\frac{1}{p}}\sup\limits_{k\geq1}|K_{a_{k}}^{\omega}(z)|
\bigg(\sum_{k=1}^{\infty}\kappa(D(a_{k},r))\bigg)^{\frac{1}{p'}}\nonumber\\
\leq& C\bigg(\sum_{k=1}^{\infty}|c_{k}|^{p}\bigg)^{\frac{1}{p}},
\end{align*}
where $r\in(0,1)$ and $C$ is a constant which only depends on $p,\rho,r$ and $\omega$.
Thus, we know that $F$ is analytic on $\mathbb{D}$. Take Lemma \ref{prop-reverse} into account, we deduce
$
\|F\|_{L_{a}^{p}(\omega)}\asymp\|F\|_{L_{a}^{p}(\kappa)}.
$
Since $(L_{a}^{p}(\kappa))^{*}\simeq L_{a}^{p'}(\kappa)$ (see \cite[Theorem D]{PRS}), together with the property of subharmonicity, similar analysis as above gives the desired result.
\end{proof}

Now we give the main result of this subsection.
\begin{proposition} \label{t-c-3} Let $0<q<\infty$, $p>\max\{2,q\}$, $\omega\in\mathcal{D}$ and $g\in H(\mathbb{D})$. Then the Volterra type
integration operator $J_{g}:L_{a}^{p}(\omega)\to H^{q}$ is compact if and only if $J_{g}$ is bounded if and only if
$g'\in T_{\frac{2p}{p-2},\widehat{\omega},\frac{2}{p-2}}^{\frac{pq}{p-q}}$. Moreover,
$\|J_{g}\|_{L_{a}^{p}(\omega)\to H^{q}}\asymp\|g'\|_{T_{\frac{2p}{p-2},\widehat{\omega},\frac{2}{p-2}}^{\frac{pq}{p-q}}}$.
\end{proposition}
\begin{proof} We divide the proof into two steps.

\textbf{Step 1.} Let $g$ be an analytic function on $\mathbb{D}$ with satisfying
\[
g'\in T_{\frac{2p}{p-2},\widehat{\omega},\frac{2}{p-2}}^{\frac{pq}{p-q}},
\]
we will show that
$J_{g}:L_{a}^{p}(\omega)\to H^{q}$ is a compact operator. Let $\{f_{j}\}_{j=1}^{\infty}$ be any sequence in $L_{a}^{p}(\omega)$ with $\|f_{j}\|_{L_{a}^{p}(\omega)}\leq 1$ and $f_{j}$ converges to 0 uniformly on any compact subset of $\mathbb{D}$. We are devoted to proving
\begin{equation}\label{eq90}
\lim\limits_{j\to\infty}\|J_{g}(f_{j})\|_{H^{q}}=0.
\end{equation}
Since $g'\in T_{\frac{2p}{p-2},\widehat{\omega},\frac{2}{p-2}}^{\frac{pq}{p-q}}$, then
for any $\varepsilon>0$, there exists $r_{0}\in(0,1)$ such that
\begin{equation}\label{eq91}
\bigg[\int_{\mathbb{T}}
\bigg(\int_{\Gamma(\xi)\backslash r_{0}\mathbb{D}}|g'(z)|^{\frac{2p}{p-2}}\bigg(\frac{(1-|z|^{2})^{2}}{\widehat{\omega}(z)}\bigg)
^{\frac{2}{p-2}}dA(z)\bigg)^{\frac{q(p-2)}{2(p-q)}}d\xi\bigg]^{\frac{p-q}{p}}<\varepsilon.
\end{equation}
Moreover, there exists $k_{0}\in\mathbb{N}$ such that for any $j>k_{0}$, $|f_{j}|^{q}<\varepsilon$ on $r_{0}\mathbb{D}$.
Combining Lemma \ref{prop-reverse}, \eqref{eq2} with Fubini's theorem, for any $j\in\mathbb{N}$, we deduce
\begin{equation}\label{r-2}
\int_{\mathbb{T}}\int_{\Gamma(\xi)\backslash r_{0}\mathbb{D}}|f_{j}(z)|^{p}\frac{\widehat{\omega}(z)}{(1-|z|^{2})^{2}}dA(z)d\xi
\lesssim \int_{\mathbb{D}}|f_{j}(z)|^{p}\frac{\widehat{\omega}(z)}{(1-|z|^{2})^{2}}\int_{I(z)}d\xi dA(z)
\asymp\|f_{j}\|_{L_{a}^{p}(\omega)}^{p}.
\end{equation}
By Lemma \ref{le1}, \eqref{eq91}, \eqref{r-2} and $\textmd{H}\ddot{\textmd{o}}\textmd{lder's}$ inequality, we obtain
\begin{align*}\label{eq3}
\|J_{g}(f_{j})\|_{H^{q}}^{q}
%&\asymp\int_{\mathbb{T}}\bigg(\int_{\Gamma(\xi)}|g'(z)|^{2}|f_{j}(z)|^{2}dA(z)\bigg)^{\frac{q}{2}}d\xi\nonumber\\
&\lesssim\int_{\mathbb{T}}\bigg(\int_{\Gamma(\xi)\bigcap r_{0}\mathbb{D}}|g'(z)|^{2}|f_{j}(z)|^{2}dA(z)\bigg)^{\frac{q}{2}}d\xi
+\int_{\mathbb{T}}\bigg(\int_{\Gamma(\xi)\backslash r_{0}\mathbb{D}}|g'(z)|^{2}|f_{j}(z)|^{2}dA(z)\bigg)^{\frac{q}{2}}d\xi\nonumber\\
&
\lesssim\varepsilon+\int_{\mathbb{T}}
\bigg[\bigg(\int_{\Gamma(\xi)\backslash r_{0}\mathbb{D}}|g'(z)|^{\frac{2p}{p-2}}\bigg(\frac{(1-|z|^{2})^{2}}{\widehat{\omega}(z)}\bigg)
^{\frac{2}{p-2}}dA(z)\bigg)^{\frac{q(p-2)}{2p}}\nonumber\\
&\cdot
\bigg(\int_{\Gamma(\xi)\backslash r_{0}\mathbb{D}}|f_{j}(z)|^{p}\frac{\widehat{\omega}(z)}{(1-|z|^{2})^{2}}dA(z)\bigg)^{\frac{q}{p}}\bigg]d\xi\nonumber\\
&\leq\varepsilon+\bigg[\int_{\mathbb{T}}
\bigg(\int_{\Gamma(\xi)\backslash r_{0}\mathbb{D}}|g'(z)|^{\frac{2p}{p-2}}\bigg(\frac{(1-|z|^{2})^{2}}{\widehat{\omega}(z)}\bigg)
^{\frac{2}{p-2}}dA(z)\bigg)^{\frac{q(p-2)}{2(p-q)}}d\xi\bigg]^{\frac{p-q}{p}}\nonumber\\
&\cdot\bigg[\int_{\mathbb{T}}\bigg(\int_{\Gamma(\xi)\backslash r_{0}\mathbb{D}}|f_{j}(z)|^{p}\frac{\widehat{\omega}(z)}{(1-|z|^{2})^{2}}dA(z)\bigg)d\xi\bigg]^{\frac{q}{p}}\nonumber\\
&\lesssim\varepsilon+\varepsilon\|f_{j}\|_{L_{a}^{p}(\omega)}^{q}\lesssim\varepsilon,\quad\textmd{for~any}~j>k_{0}.
\end{align*}

\textbf{Step 2.} Assume that $J_{g}$ is bounded. For $r\in(0,1)$, let $\tilde{r}=\tanh r$ and $\{a_{k}\}_{k=1}^{\infty}$ be an $\tilde{r}$-lattice.
It follows from Lemma \ref{D-regular} that
\begin{equation}\label{eq81}
\widehat{\omega}(z)\asymp\widehat{\omega}(a_{k}),\quad z\in D(a_{k},\tilde{r}).
\end{equation}
By the subharmonicity of $|g'|^{2}$ and \eqref{eq81}, for any $\xi\in\mathbb{T}$, we get
\begin{align*}
\int_{\Gamma(\xi)}|g'(z)|^{\frac{2p}{p-2}}
\frac{(1-|z|^{2})^{\frac{4}{p-2}}}{\widehat{\omega}(z)^{\frac{2}{p-2}}}dA(z)
&\lesssim \int_{\Gamma(\xi)}\bigg(\int_{D(z,\tilde{r})}|g'(\zeta)|^{2}dA(\zeta)\bigg)^{\frac{p}{p-2}}
\frac{1}{\widehat{\omega}(z)^{\frac{2}{p-2}}(1-|z|^{2})^{2}}dA(z)\\
&\leq\sum_{k}\int_{D(a_{k},\tilde{r})}\Big(\int_{D(z,\tilde{r})}|g'(\zeta)|^{2}dA(\zeta)\Big)^{\frac{p}{p-2}}
\frac{1}{\widehat{\omega}(z)^{\frac{2}{p-2}}(1-|z|^{2})^{2}}dA(z)\\
&\lesssim \sum_k\Big(\sup_{\zeta\in D(a_{k},2\tilde{r})}|g'(\zeta)|^{\frac{2p}{p-2}}\Big)
\frac{(1-|a_{k}|^{2})^{\frac{2p}{p-2}}}{\widehat{\omega}(a_{k})^{\frac{2}{p-2}}},
\end{align*}
where the summation runs $\{k:D(a_{k},\tilde{r})\bigcap\Gamma(\xi)\neq\emptyset\}$.
Let $\delta=3e^{2r}-2$. By Lemma \ref{o-g},
for any $\xi\in\mathbb{T}$, we have
\[
D(\zeta,\tilde{r})\subseteq \Gamma_{\delta}(\xi),\quad \forall\zeta\in \Gamma(\xi).
\]
Hence, for any $\xi\in\mathbb{T}$, we deduce
\begin{align*}
\int_{\Gamma(\xi)}|g'(z)|^{\frac{2p}{p-2}}
\frac{(1-|z|^{2})^{\frac{4}{p-2}}}{\widehat{\omega}(z)^{\frac{2}{p-2}}}dA(z)
\lesssim \sum_{a_{k}\in\Gamma_{\delta}(\xi)}
\bigg(\sup_{\zeta\in D(a_{k},2\tilde{r})}|g'(\zeta)|^{\frac{2p}{p-2}}\bigg)
\frac{(1-|a_{k}|^{2})^{\frac{2p}{p-2}}}{\widehat{\omega}(a_{k})^{\frac{2}{p-2}}}.
\end{align*}
Thus, we obtain
\begin{eqnarray*}\label{b-t-eq19}\begin{split}
&~~~~
\int_{\mathbb{T}}\bigg(\int_{\Gamma(\xi)}|g'(z)|^{\frac{2p}{p-2}}
\frac{(1-|z|^{2})^{\frac{4}{p-2}}}{\widehat{\omega}(z)^{\frac{2}{p-2}}}dA(z)\bigg)^{\frac{q(p-2)}{2(p-q)}}d\xi\\&
\lesssim\int_{\mathbb{T}}\bigg(\sum_{a_{k}\in\Gamma(\xi)}
\bigg(\sup_{\zeta\in D(a_{k},2\tilde{r})}|g'(\zeta)|^{\frac{2p}{p-2}}\bigg)
\frac{(1-|a_{k}|^{2})^{\frac{2p}{p-2}}}{\widehat{\omega}(a_{k})^{\frac{2}{p-2}}}\bigg)^{\frac{q(p-2)}{2(p-q)}}d\xi.
\end{split}\end{eqnarray*}
In what follows, we only need to prove
\begin{equation}\label{eq4}
\int_{\mathbb{T}}\bigg(\sum\limits_{a_{k}\in\Gamma(\xi)}
\bigg(\sup_{\zeta\in D(a_{k},2\tilde{r})}|g'(\zeta)|^{\frac{2p}{p-2}}\bigg)\frac{(1-|a_{k}|^{2})^{\frac{2p}{p-2}}}
{\widehat{\omega}(a_{k})^{\frac{2}{p-2}}}\bigg)^{\frac{q(p-2)}{2(p-q)}}d\xi\lesssim\|J_{g}\|_{L_{a}^{p}(\omega)\to H^{q}}^{\frac{pq}{p-q}}.
\end{equation}
Define a sequence $\{b_{k}\}_{k=1}^{\infty}$ by
\[
b_{k}=\bigg(\sup_{\zeta\in D(a_{k},2\tilde{r})}|g'(\zeta)|^{q}\bigg)\frac{(1-|a_{k}|^{2})^{q}}{\widehat{\omega}(a_{k})^{\frac{q}{p}}}.
\]
Let
$\eta=\frac{2p}{q(p-2)}$. We turn to prove that
there exists a positive constant $s$ such that
\begin{equation}\label{eq14}
\|\{b_{k}^{\frac{1}{s}}\}\|_{T_{\eta s}^{\frac{ps}{p-q}}(\{a_{k}\})}\lesssim\|J_{g}\|_{L_{a}^{p}(\omega)\to H^{q}}^{\frac{q}{s}}.
\end{equation}
Choose a constant $s$ with $s>\max\{1,\frac{q}{2}\}$. From Lemma \ref{le2}, the dual space of
$T_{(\eta s)'}^{\big(\frac{ps}{p-q}\big)'}(\{a_{k}\})$ can be identified with $T_{\eta s}^{\frac{ps}{p-q}}(\{a_{k}\})$.
Thus, by \eqref{eq2}, for any sequence
$\{e_{k}\}\in T_{(\eta s)'}^{\big(\frac{ps}{p-q}\big)'}(\{a_{k}\})$, it holds that
\begin{equation}\label{eq60}
|\langle e_{k},b_{k}^{\frac{1}{s}}\rangle_{T_{2}^{2}(\{a_{k}\})}|\leq\sum\limits_{k=1}^{\infty}|e_{k}b_{k}^{\frac{1}{s}}|(1-|a_{k}|^{2})
\asymp \sum\limits_{k=1}^{\infty}|e_{k}b_{k}^{\frac{1}{s}}|\int_{I(a_{k})}d\xi
=\int_{\mathbb{T}}\sum\limits_{a_{k}\in\Gamma(\xi)}|e_{k}b_{k}^{\frac{1}{s}}|d\xi.
\end{equation}
By Lemma \ref{le4}, we obtain
\begin{equation}\label{eq12}
T_{(\eta s)'}^{\big(\frac{ps}{p-q}\big)'}(\{a_{k}\})=T_{\frac{2s}{2s-q}}^{s'}(\{a_{k}\})\cdot T_{\frac{ps}{q}}^{\frac{ps}{q}}(\{a_{k}\}),
\end{equation}
from which there exist two sequences $\{x_{k}\}\in T_{\frac{2s}{2s-q}}^{s'}(\{a_{k}\})$ and
$\{y_{k}\}\in T_{p}^{p}(\{a_{k}\})$ such that
\[
e_{k}=x_{k}y_{k}^{\frac{q}{s}}.
\]
Together with \eqref{eq60} and $\textmd{H}\ddot{\textmd{o}}\textmd{lder's}$ inequality, we get
\begin{align}\label{eq13}
|\langle e_{k},b_{k}^{\frac{1}{s}}\rangle_{T_{2}^{2}(\{a_{k}\})}|
&\lesssim\int_{\mathbb{T}}\bigg(\sum\limits_{a_{k}\in\Gamma(\xi)}|x_{k}|^{\frac{2s}{2s-q}}\bigg)^{\frac{2s-q}{2s}}
\bigg(\sum\limits_{a_{k}\in\Gamma(\xi)}|y_{k}^{2}b_{k}^{\frac{2}{q}}|\bigg)^{\frac{q}{2s}}d\xi\nonumber\\
&\leq \|\{x_{k}\}\|_{T_{\frac{2s}{2s-q}}^{s'}(\{a_{k}\})}
\bigg[\int_{\mathbb{T}}\bigg(\sum\limits_{a_{k}\in\Gamma(\xi)}|y_{k}^{2}b_{k}^{\frac{2}{q}}|\bigg)^{\frac{q}{2}}d\xi\bigg]^{\frac{1}{s}}.
\end{align}
Let
\begin{equation}\label{eq77}
f_{k}(z)=\frac{K_{a_{k}}^{\omega}(z)}{\|K_{a_{k}}^{\omega}\|_{L_{a}^{2-\frac{1}{p}}(\omega)}^{2-\frac{1}{p}}},\quad z\in\mathbb{D}
\end{equation}
and $\{r_{k}:k\in\mathbb{N}\}$ be the set of the Rademacher functions.
For $t\in[0,1]$, set
\[
F_{t}(z)=\sum_{k=1}^{\infty}(1-|a_{k}|^{2})^{\frac{1}{p}}y_{k}r_{k}(t)f_{k}(z),\quad z\in\mathbb{D}.
\]
Then
\begin{equation}\label{eq78}
\|\{y_{k}\}\|_{T_{p}^{p}(\{a_{k}\})}^{p}
=\int_{\mathbb{T}}\sum_{a_{k}\in\Gamma(\xi)}|y_{k}|^{p}d\xi
=\sum\limits_{k=1}^{\infty}\int_{I(a_{k})}|y_{k}|^{p}d\xi
\asymp\sum_{k=1}^{\infty}|y_{k}|^{p}(1-|a_{k}|^{2}).
\end{equation}
By using Proposition \ref{le5} and \eqref{eq78}, we obtain
\begin{equation*}\label{eq7}
\|J_{g}(F_{t})\|_{H^{q}}^{q}
\lesssim\|J_{g}\|_{L_{a}^{p}(\omega)\to H^{q}}^{q}\big(\sum_{k=1}^{\infty}|y_{k}|^{p}(1-|a_{k}|^{2})\big)^{\frac{q}{p}}
\asymp\|J_{g}\|_{L_{a}^{p}(\omega)\to H^{q}}^{q}\|\{y_{k}\}\|_{T_{p}^{p}(\{a_{k}\})}^{q}.
\end{equation*}
Together with Lemma \ref{le1}, we deduce
\begin{equation}\label{eq18}
\int_{\mathbb{T}}\bigg(\int_{\Gamma(\xi)}|g'(z)|^{2}
\bigg|\sum_{k=1}^{\infty}(1-|a_{k}|^{2})^{\frac{1}{p}}y_{k}r_{k}(t)f_{k}(z)\bigg|^{2}dA(z)\bigg)^{\frac{q}{2}}d\xi
\lesssim\|J_{g}\|_{L_{a}^{p}(\omega)\to H^{q}}^{q}\cdot\|\{y_{k}\}\|_{T_{p}^{p}(\{a_{k}\})}^{q}.
\end{equation}
Integrate both sides of the inequality \eqref{eq18} from 0 to 1 with respect to $t$, by Fubini's theorem, we have
\begin{equation}\label{eq8}
\int_{\mathbb{T}}\int_{0}^{1}\bigg(\int_{\Gamma(\xi)}|g'(z)|^{2}
\bigg|\sum_{k=1}^{\infty}(1-|a_{k}|^{2})^{\frac{1}{p}}y_{k}r_{k}(t)f_{k}(z)\bigg|^{2}dA(z)\bigg)^{\frac{q}{2}}dtd\xi\lesssim
 \|J_{g}\|_{L_{a}^{p}(\omega)\to H^{q}}^{q}\cdot\|\{y_{k}\}\|_{T_{p}^{p}(\{a_{k}\})}^{q}.
\end{equation}
By the subharmonicity of $|g'^{2}|$ and Lemma \ref{D-regular}, for any $k\in\mathbb{N}$, we obtain
\begin{align}\label{eq58}
\sup\limits_{\zeta\in D(a_{k},2\tilde{r})}|g'(\zeta)|^{2}
&\lesssim\sup\limits_{\zeta\in D(a_{k},2\tilde{r})}
\frac{1}{(1-|\zeta|)^{2}}\int_{B(\zeta,\tilde{r}(1-|\zeta|))}|g'(z)|^{2}dA(z)\nonumber\\
&\leq\sup\limits_{\zeta\in D(a_{k},2\tilde{r})}
\frac{1}{(1-|\zeta|)^{2}}\int_{D(\zeta,\tilde{r})}|g'(z)|^{2}dA(z)\nonumber\\
&\leq\frac{1}{\widehat{\omega}(a_{k})(1-|a_{k}|)}\int_{D(a_{k},3\tilde{r})}|g'(z)|^{2}\frac{\widehat{\omega}(z)}{(1-|z|)}dA(z).
\end{align}
For $z\in D(a_{k},3\tilde{r})$, it follows from \cite[Lemma 8]{PRS} and \eqref{reproducing-formula} that for $r$ small enough, we have
\begin{equation}\label{eq59}
|f_{k}(z)|\asymp \big[\widehat{\omega}(a_{k})(1-|a_{k}|)\big]^{1-\frac{1}{p}}|K_{a_{k}}^{\omega}(z)|\asymp \frac{1}{\widehat{\omega}(a_{k})^{\frac{1}{p}}(1-|a_{k}|^{2})^{\frac{1}{p}}}.
\end{equation}
Take $\iota>\max\{1,\frac{2}{q}\}$. Combining Lemma \ref{D-regular}, \eqref{eq58} with \eqref{eq59}, for $\xi\in\mathbb{T}$, we deduce
\begin{align*}\label{r-5}
\sum\limits_{a_{k}\in\Gamma(\xi)}|y_{k}^{2}b_{k}^{\frac{2}{q}}|
&\asymp\sum\limits_{a_{k}\in\Gamma(\xi)}|y_{k}|^{2}
\sup\limits_{\zeta\in D(a_{k},2\tilde{r})}|g'(\zeta)|^{2}\bigg(\frac{1-|a_{k}|^{2}}{|1-\bar{a_{k}}\xi|}\bigg)^{\iota}\frac{(1-|a_{k}|^{2})^{2}}
{\widehat{\omega}(a_{k})^{\frac{2}{p}}}\nonumber\\
&\lesssim\sum\limits_{k=1}^{\infty}|y_{k}|^{2}
\frac{1}{\widehat{\omega}(a_{k})(1-|a_{k}|)}\int_{D(a_{k},3\tilde{r})}|g'(z)|^{2}\frac{\widehat{\omega}(z)}{1-|z|}
\bigg(\frac{1-|z|^{2}}{|1-\bar{z}\xi|}\bigg)^{\iota}
\frac{(1-|z|^{2})^{2}}{\widehat{\omega}(z)^{\frac{2}{p}}}dA(z)\nonumber\\
&\asymp\sum\limits_{k=1}^{\infty}|y_{k}|^{2}(1-|a_{k}|^{2})^{\frac{2}{p}}
\int_{D(a_{k},3\tilde{r})}|g'(z)|^{2}\bigg(\frac{1-|z|^{2}}{|1-\bar{z}\xi|}\bigg)^{\iota}|f_{k}(z)|^{2}dA(z)\nonumber\\
&\leq\int_{\mathbb{D}}\bigg(\frac{1-|z|^{2}}{|1-\bar{z}\xi|}\bigg)^{\iota}
\sum\limits_{k=1}^{\infty}|y_{k}|^{2}(1-|a_{k}|^{2})^{\frac{2}{p}}|f_{k}(z)|^{2}|g'(z)|^{2}dA(z).
\end{align*}
Together with Lemma \ref{le3}, we know
\begin{align*}
\int_{\mathbb{T}}\bigg(\sum\limits_{a_{k}\in\Gamma(\xi)}|y_{k}^{2}b_{k}^{\frac{2}{q}}|\bigg)^{\frac{q}{2}}d\xi
&\lesssim\int_{\mathbb{T}}\bigg(\int_{\mathbb{D}}\bigg(\frac{1-|z|^{2}}{|1-\bar{z}\xi|}\bigg)^{\iota}
\sum\limits_{k=1}^{\infty}|y_{k}|^{2}(1-|a_{k}|^{2})^{\frac{2}{p}}|f_{k}(z)|^{2}|g'(z)|^{2}dA(z)\bigg)^{\frac{q}{2}}d\xi\nonumber\\
&\asymp\int_{\mathbb{T}}\bigg(\int_{\Gamma(\xi)}\sum\limits_{k=1}^{\infty}|y_{k}|^{2}(1-|a_{k}|^{2})
^{\frac{2}{p}}|f_{k}(z)|^{2}|g'(z)|^{2}dA(z)\bigg)^{\frac{q}{2}}d\xi.
\end{align*}
Moreover, using Lemma \ref{le-Khinchine}, Lemma \ref{le-Kahane} and Fubini's theorem, we deduce
\begin{eqnarray*}\label{eq9}\begin{split}
&~~~~\int_{\mathbb{T}}\bigg(\int_{\Gamma(\xi)}\sum\limits_{k=1}^{\infty}|y_{k}|^{2}(1-|a_{k}|^{2})
^{\frac{2}{p}}|f_{k}(z)|^{2}|g'(z)|^{2}dA(z)\bigg)^{\frac{q}{2}}d\xi\\&
\asymp\int_{\mathbb{T}}\bigg(\int_{\Gamma(\xi)}|g'(z)|^{2}\int_{0}^{1}\bigg|\sum\limits_{k=1}^{\infty}y_{k}r_{k}(t)
(1-|a_{k}|^{2})^{\frac{1}{p}}f_{k}(z)\bigg|^{2}dtdA(z)\bigg)^{\frac{q}{2}}d\xi\\&
=\int_{\mathbb{T}}\bigg(\int_{0}^{1}\int_{\Gamma(\xi)}\bigg|g'(z)\sum\limits_{k=1}^{\infty}y_{k}r_{k}(t)
(1-|a_{k}|^{2})^{\frac{1}{p}}f_{k}(z)\bigg|^{2}dA(z)dt\bigg)^{\frac{q}{2}}d\xi\\&
\asymp\int_{\mathbb{T}}\int_{0}^{1}\bigg(\int_{\Gamma(\xi)}\bigg|g'(z)\sum\limits_{k=1}^{\infty}y_{k}r_{k}(t)
(1-|a_{k}|^{2})^{\frac{1}{p}}f_{k}(z)\bigg|^{2}dA(z)\bigg)^{\frac{q}{2}}dtd\xi.
\end{split}\end{eqnarray*}
Combining with \eqref{eq8}, we conclude that
\begin{equation}\label{eq11}
\int_{\mathbb{T}}\bigg(\sum\limits_{a_{k}\in\Gamma(\xi)}|y_{k}^{2}b_{k}^{\frac{2}{q}}|\bigg)^{\frac{q}{2}}d\xi
\lesssim\|J_{g}\|_{L_{a}^{p}(\omega)\to H^{q}}^{q}\cdot\|\{y_{k}\}\|_{T_{p}^{p}(\{a_{k}\})}^{q}.
\end{equation}
According to \eqref{eq13} and \eqref{eq11}, we deduce
\begin{equation}\label{eq61}
|\langle e_{k},b_{k}^{\frac{1}{s}}\rangle_{T_{2}^{2}(\{a_{k}\})}|
\lesssim\|\{x_{k}\}\|_{T_{\frac{2s}{2s-q}}^{s'}(\{a_{k}\})}\cdot
\|J_{g}\|_{L_{a}^{p}(\omega)\to H^{q}}^{\frac{q}{s}}\cdot\|\{y_{k}\}\|_{T_{p}^{p}(\{a_{k}\})}^{\frac{q}{s}}.
\end{equation}
Take infimum over all possible factorizations of $\{e_{k}\}$ as in \eqref{eq12}, we deduce
\[
\inf\|\{x_{k}\}\|_{T_{\frac{2s}{2s-q}}^{s'}(\{a_{k}\})}\cdot\|\{y_{k}\}\|_{T_{p}^{p}(\{a_{k}\})}^{\frac{q}{s}}
\lesssim \|\{e_{k}\}\|_{T_{(\eta s)'}^{\big(\frac{ps}{p-q}\big)'}(\{a_{k}\})}.
\]
Combining with \eqref{eq61}, we get
\[
|\langle e_{k},b_{k}^{\frac{1}{s}}\rangle_{T_{2}^{2}(\{a_{k}\})}|
\lesssim\|J_{g}\|_{L_{a}^{p}(\omega)\to H^{q}}^{\frac{q}{s}}\cdot\|\{e_{k}\}\|_{T_{(\eta s)'}^{\big(\frac{ps}{p-q}\big)'}(\{a_{k}\})}.
\]
Finally, it follows from Lemma \ref{le2} that \eqref{eq14} holds. Thus, we get the result as desired.
\end{proof}
\subsection{Case D}
Last, we will deal with the remaining case $0<q<p\leq2$, with the main result as follows.
\begin{proposition} \label{t4} Let $0<q<p\leq2$, $\omega\in\mathcal{D}$ and $g\in H(\mathbb{D})$.
Then the Volterra type
integration operator $J_{g}:L_{a}^{p}(\omega)\to H^{q}$ is compact if and only if
\[
\lim\limits_{r\to1^{-}}\int_{\mathbb{T}}
\sup_{z\in\Gamma(\xi)\backslash r\mathbb{D}}
|g'(z)|^{\frac{pq}{p-q}}\frac{(1-|z|^{2})^{\frac{pq}{p-q}}}{\widehat{\omega}(z)^{\frac{q}{p-q}}}d\xi=0.
\]
\end{proposition}
\begin{proof}
Sufficiency. Let $\{f_{j}\}_{j=1}^{\infty}$ be any sequence in $L_{a}^{p}(\omega)$ with $\|f_{j}\|_{L_{a}^{p}(\omega)}\leq 1$ and $f_{j}$ converges to 0 uniformly on any compact subset of $\mathbb{D}$. Then we will prove $\|J_{g}(f_{j})\|_{H^{q}}\to 0$ as $j\to\infty$.
By the assumption, it follows that for any $\varepsilon>0$, there exists $R_{0}\in(0,1)$ such that for any $R_{1}\in[R_{0},1)$, it holds that
\begin{equation} \label{b-t-eq1}
\int_{\mathbb{T}}\sup_{z\in\Gamma(\xi)\backslash R_{1}\mathbb{D}}
|g'(z)|^{\frac{pq}{p-q}}\frac{(1-|z|^{2})^{\frac{pq}{p-q}}}{\widehat{\omega}(z)^{\frac{q}{p-q}}}d\xi
<\varepsilon^{\frac{p}{p-q}}.
\end{equation}
Note that there also exists $k_{0}\in\mathbb{N}$ such that
\begin{equation} \label{b-t-eq2}
|f_{j}|<\varepsilon^{\frac{1}{q}}~\textmd{on}~R_{0}\mathbb{D},\quad \textmd{for~any}~j>k_{0}.
\end{equation}
Choose $\gamma\in(0,\infty)$ with $\tilde{\gamma}:=\tanh \gamma\in(0,\frac{1}{2})$ and $\{a_{k}\}_{k=1}^{\infty}$ be a $\tilde{\gamma}$-lattice. Then we know
\[
\Gamma(\xi)\subseteq\bigcup_{\{a_{k}:D(a_{k},\tilde{\gamma})\bigcap\Gamma(\xi)\neq\emptyset\}}D(a_{k},\tilde{\gamma}).
\]
Take $\alpha=3e^{2\gamma}-2$.
By Lemma \ref{o-g}, we get
\[
D(\zeta,\tilde{\gamma})\subseteq\Gamma_{\alpha}(\xi),\quad\zeta\in\Gamma(\xi).
\]
Hence, we obtain
\begin{equation} \label{eq23}
\Gamma(\xi)\subseteq \bigcup_{a_{k}\in\Gamma_{\alpha}(\xi)}D(a_{k},\tilde{\gamma}).
\end{equation}
Observe (\cite{Pau-2020}) that there exists a constant $\beta>\alpha$
such that
\begin{equation} \label{eq67}
\bigcup_{a_{k}\in\Gamma_{\alpha}(\xi)}D(a_{k},\tilde{\gamma})\subseteq \Gamma_{\beta}(\xi),\quad \xi\in\mathbb{T}.
\end{equation}
For any $j\in\mathbb{N}$, by the subharmonicity of $|f_{j}|^{2}$, we deduce
\begin{align}\label{eq68}
|f_{j}(z)|^{2}
&\lesssim\bigg(\frac{1}{(1-|a_{k}|)^{2}}\int_{D(a_{k},2\tilde{\gamma})}|f_{j}(u)|^{p}dA(u)\bigg)^{\frac{2}{p}}, \quad z\in D(a_{k},\tilde{\gamma}).
\end{align}
It follows from Lemma \ref{D-regular} that
\begin{equation} \label{p-a-1}
\widehat{\omega}(z)\asymp\widehat{\omega}(a_{k}),\quad z\in D(a_{k},2\tilde{\gamma}).
\end{equation}
According to \cite[Theorem 3.3]{zhu2007}, if $p\in(0,1]$, then for any sequence $\{z_{k}\}_{k=1}^{\infty}\in l^{p}$,
\begin{equation}\label{h-p}
\bigg|\sum_{k=1}^{\infty}z_{k}\bigg|\leq\bigg(\sum_{k=1}^{\infty}|z_{k}|^{p}\bigg)^{\frac{1}{p}}.
\end{equation}
Combining \eqref{eq23}, \eqref{eq67}, \eqref{eq68}, \eqref{p-a-1} with \eqref{h-p}, for any $j\in\mathbb{N}$, we obtain
\begin{equation}\label{r-3}
\bigg(\int_{\Gamma(\xi)}|f_{j}(z)|^{2}(1-|z|^{2})^{-2}\widehat{\omega}(z)^{\frac{2}{p}}dA(z)\bigg)^{\frac{p}{2}}
\lesssim\int_{\Gamma_{\beta}(\xi)}|f_{j}(u)|^{p}\frac{\widehat{\omega}(u)}{(1-|u|)^{2}}dA(u).
\end{equation}
By Lemma \ref{prop-reverse}, \eqref{eq2} and Fubini's theorem, for any $j\in\mathbb{N}$, we have
\begin{equation*}\label{r-4}
\int_{\mathbb{T}}\int_{\Gamma_{\beta}(\xi)}|f_{j}(u)|^{p}\frac{\widehat{\omega}(u)}{(1-|u|)^{2}}dA(u)d\xi
=\int_{\mathbb{D}}|f_{j}(u)|^{p}\frac{\widehat{\omega}(u)}{(1-|u|)^{2}}\int_{I_{\beta}(u)}d\xi dA(u)
\asymp\|f_{j}\|_{L_{a}^{p}(\omega)}^{p}.
\end{equation*}
Together with Lemma \ref{le1}, \eqref{b-t-eq1}, \eqref{b-t-eq2}, \eqref{r-3} and $\textmd{H}\ddot{\textmd{o}}\textmd{lder's}$ inequality, for any $j>k_{0}$, we get
\begin{align*}\label{eq21}
\|J_{g}(f_{j})\|_{H^{q}}^{q}
&\lesssim\int_{\mathbb{T}}\bigg(\int_{\Gamma(\xi)\bigcap R_{0}\mathbb{D}}|g'(z)|^{2}|f_{j}(z)|^{2}dA(z)\bigg)^{\frac{q}{2}}d\xi
+\int_{\mathbb{T}}\bigg(\int_{\Gamma(\xi)\backslash R_{0}\mathbb{D}}|g'(z)|^{2}|f_{j}(z)|^{2}dA(z)\bigg)^{\frac{q}{2}}d\xi\nonumber\\
&\lesssim\varepsilon+\int_{\mathbb{T}}
\bigg(\sup\limits_{z\in\Gamma(\xi)\backslash R_{0}\mathbb{D}}|g'(z)|^{q}\frac{(1-|z|^{2})^{q}}{\widehat{\omega}(z)^{\frac{q}{p}}}\bigg)\cdot
\bigg(\int_{\Gamma(\xi)}|f_{j}(z)|^{2}\frac{\widehat{\omega}(z)^{\frac{2}{p}}}{(1-|z|^{2})^{2}}dA(z)\bigg)^{\frac{q}{2}}d\xi\nonumber\\
&\leq \varepsilon+
\bigg[\int_{\mathbb{T}}\bigg(\sup_{z\in\Gamma(\xi)\backslash R_{0}\mathbb{D}}
|g'(z)|^{q}\frac{(1-|z|^{2})^{q}}{\widehat{\omega}(z)^{\frac{q}{p}}}\bigg)^{\frac{p}{p-q}}d\xi\bigg]^{\frac{p-q}{p}}\nonumber\\
&\cdot\bigg[\int_{\mathbb{T}}\bigg(\int_{\Gamma(\xi)}|f_{j}(z)|^{2}\frac{\widehat{\omega}(z)^{\frac{2}{p}}}{(1-|z|^{2})^{2}}dA(z)\bigg)
^{\frac{p}{2}}d\xi\bigg]^{\frac{q}{p}}\nonumber\\
&\lesssim\varepsilon+\varepsilon\|f_{j}\|_{L_{a}^{p}(\omega)}^{q}\lesssim\varepsilon,
\end{align*}
from which it follows that $J_{g}:L_{a}^{p}(\omega)\to H^{q}$ is compact.
\vskip 0.1in
Necessity. We divide the proof into two steps.
\vskip 0.1in

\textbf{Step 1.} Let $\delta$ be a positive constant such that
$\tilde{\delta}:=\tanh \delta\in(0,\frac{1}{3})$.
Take $\{a_{k}\}_{k=1}^{\infty}$ be a $\tilde{\delta}$-lattice. Rearrange the sequence $\{a_{k}\}_{k=1}^{\infty}$ such that
$|a_{1}|\leq |a_{2}|\leq\cdots\to1^{-}$.
For any $\rho_{0}\in(0,1)$ and $\xi\in\mathbb{T}$, we know that
\[
\Gamma(\xi)\backslash\rho_{0}\mathbb{D}\subseteq\bigcup_{\{a_{k}:D(a_{k},\tilde{\delta})\cap(\Gamma(\xi)\backslash\rho_{0}\mathbb{D})\neq\emptyset\}}
D(a_{k},\tilde{\delta}).
\]
Let $\lambda=3e^{2\delta}-2$. By Lemma \ref{o-g},
for any $\xi\in\mathbb{T}$, we have
\[
D(\zeta,\tilde{\delta})\subseteq \Gamma_{\lambda}(\xi),\quad \textmd{for~any}~\zeta\in \Gamma(\xi).
\]
Hence, for any $\zeta\in D(a_{k},\tilde{\delta})\cap(\Gamma(\xi)\backslash\rho_{0}\mathbb{D})$, there exists $\rho_{1}\in(0,\rho_{0}]$ such that
\begin{equation}\label{b-t-eq8}
\Gamma(\xi)\backslash\rho_{0}\mathbb{D}\subseteq \bigcup_{a_{k}\in\Gamma_{\lambda}(\xi)\backslash \rho_{1}\mathbb{D}}D(a_{k},\tilde{\delta})
\end{equation}
and $\rho_{1}$ tends to 1 as $\rho_{0}$ tends to $1$.
Recall \cite[Proposition 4.5]{zhu2007}, for any $z\in D(a_{k},\tilde{\delta})$, we know that
\[
\frac{1-\tilde{\delta}}{1+\tilde{\delta}}(1-|z|^{2})\leq 1-|a_{k}|^{2}\leq \frac{1+\tilde{\delta}}{1-\tilde{\delta}}(1-|z|^{2})
\]
and
\begin{equation}\label{b-t-eq11}
A(D(a_{k},\tilde{\delta}))=\frac{\tilde{\delta}^{2}(1-|a_{k}|^{2})^{2}}{(1-|a_{k}|^{2}\tilde{\delta}^{2})^{2}}.
\end{equation}
Recall (see \cite{PRS}) there exist positive constants $\varrho_{1}$ and $\varrho_{2}$ such that
\begin{equation}\label{b-t-eq12}
\big(\frac{1-t_{1}}{1-t_{2}}\big)^{\varrho_{1}}\widehat{\omega}(t_{2})
\lesssim\widehat{\omega}(t_{1})\lesssim\big(\frac{1-t_{1}}{1-t_{2}}\big)^{\varrho_{2}}\widehat{\omega}(t_{2}), \quad0\leq t_{1}\leq t_{2}<1.
\end{equation}
By the subharmonicity of $|g'|$, for any $z\in D(a_{k},\tilde{\delta})$, we have
\begin{align}\label{b-t-eq13}
|g'(z)|
&\leq\frac{1}{\tilde{\delta}^{2}(1-|z|)^{2}}\int_{B(z,\tilde{\delta}(1-|z|))}|g'(u)|dA(u)\nonumber\\
&\leq\frac{1}{\tilde{\delta}^{2}(1-|z|)^{2}}\int_{D(z,\tilde{\delta})}|g'(u)|dA(u)\nonumber\\
&\lesssim\frac{1}{\tilde{\delta}^{2}(1-|a_{k}|)^{2}}\int_{D(a_{k},2\tilde{\delta})}|g'(u)|dA(u).
\end{align}
Combining \eqref{b-t-eq8}, \eqref{b-t-eq11}, \eqref{b-t-eq12} with \eqref{b-t-eq13}, we obtain
\begin{align*}\label{eq38}
\sup\limits_{z\in\Gamma(\xi)\backslash\rho_{0}\mathbb{D}}\frac{(1-|z|^{2})}{\widehat{\omega}(z)^{\frac{1}{p}}}|g'(z)|
&\lesssim\sup\limits_{a_{k}\in\Gamma_{\lambda}(\xi)\backslash \rho_{1}\mathbb{D}}
\frac{1}{\tilde{\delta}^{2}(1-|a_{k}|^{2})\widehat{\omega}(a_{k})^{\frac{1}{p}}}\int_{D(a_{k},2\tilde{\delta})}|g'(u)|dA(u)\nonumber\\
&\lesssim\sup\limits_{a_{k}\in\Gamma_{\lambda}(\xi)\backslash \rho_{1}\mathbb{D}}
\frac{(1-|a_{k}|^{2})}{\widehat{\omega}(a_{k})^{\frac{1}{p}}}\bigg(\sup\limits_{u\in D(a_{k},2\tilde{\delta})}|g'(u)|\bigg).
\end{align*}
Notice that the constants of comparison in the above inequality do not depend on $\tilde{\delta}, \rho_{0}$ and $\rho_{1}$.
Then we deduce there exists a constant $C>0$ ($C$ does not depend on the choices of $\tilde{\delta}, \rho_{0}$ and $\rho_{1}$) such that
\begin{align*}
\bigg(\sup\limits_{z\in\Gamma(\xi)\backslash\rho_{0}\mathbb{D}}\frac{(1-|z|^{2})}{\widehat{\omega}(z)^{\frac{1}{p}}}|g'(z)|\bigg)^{\frac{pq}{p-q}}
&\leq C\bigg[\sup\limits_{a_{k}\in\Gamma_{\lambda}(\xi)\backslash \rho_{1}\mathbb{D}}
\frac{(1-|a_{k}|^{2})}{\widehat{\omega}(a_{k})^{\frac{1}{p}}}\bigg(\sup\limits_{u\in D(a_{k},2\tilde{\delta})}|g'(u)|\bigg)\bigg]^{\frac{pq}{p-q}}.
\end{align*}
Thus, by integrating over $\mathbb{T}$ with respect to $\xi$, we conclude that
\begin{equation}\label{b-t-eq14}
\int_{\mathbb{T}}\bigg(\sup\limits_{z\in\Gamma(\xi)\backslash\rho_{0}\mathbb{D}}\frac{(1-|z|^{2})}
{\widehat{\omega}(z)^{\frac{1}{p}}}|g'(z)|\bigg)^{\frac{pq}{p-q}}d\xi
\leq C\int_{\mathbb{T}}\bigg[\sup\limits_{a_{k}\in\Gamma_{\lambda}(\xi)\backslash \rho_{1}\mathbb{D}}
\frac{(1-|a_{k}|^{2})}{\widehat{\omega}(a_{k})^{\frac{1}{p}}}
\bigg(\sup\limits_{u\in D(a_{k},2\tilde{\delta})}|g'(u)|\bigg)\bigg]^{\frac{pq}{p-q}}d\xi.
\end{equation}

\textbf{Step 2.} We will prove that
\begin{equation}\label{eq27}
\lim\limits_{r\to1^{-}}\int_{\mathbb{T}}
\bigg[\sup\limits_{a_{k}\in\Gamma(\xi)\backslash r\mathbb{D}}\frac{(1-|a_{k}|^{2})}{\widehat{\omega}(a_{k})^{\frac{1}{p}}}
\bigg(\sup\limits_{u\in D(a_{k},2\tilde{\delta})}|g'(u)|\bigg)\bigg]^{\frac{pq}{p-q}}d\xi=0.
\end{equation}

Note that for the any $\varepsilon>0$, there exists $R_{2}\in(0,1)$ such that for any $R_{3}\in(R_{2},1)$, it holds that
\begin{equation}\label{b-t-eq3}
\widehat{\omega}(R_{3})(1-R_{3})<\varepsilon.
\end{equation}
Moreover there exists $k_{1}\in\mathbb{N}$ such that for any $k>k_{1}$, it holds that $|a_{k}|\in(R_{2},1)$. Combining with \eqref{b-t-eq3}, we have
\begin{equation}\label{b-t-eq4}
\widehat{\omega}(a_{k})(1-|a_{k}|)<\varepsilon,\quad\textmd{for~any}~k>k_{1}.
\end{equation}
Since every $z\in\mathbb{D}$ belongs to at most $N$ ($N$ is a constant only depends on $\delta$) $D(a_{k},\tilde{\delta})$, then we get
\[
\sum_{k=1}^{\infty}\kappa(D(a_{k},\tilde{\delta}))\leq N\kappa(\mathbb{D}),
\]
where
\[
\kappa(z)=\frac{\widehat{\omega}(z)}{1-|z|},\quad z\in\mathbb{D}.
\]
Hence, there exists $k_{2}\in\mathbb{N}$ such that for any $k_{3}>k_{2}$,
\begin{equation}\label{b-t-eq5}
\sum_{k=k_{3}}^{\infty}\kappa(D(a_{k},\tilde{\delta}))<\varepsilon.
\end{equation}
Let $r_{k}$ be the Rademacher functions.
For $k\in\mathbb{N}$ and $p\in(0,2]$, consider the functions
\[
f_{k,p}(z)=\frac{K_{a_{k}}^{\omega}(z)^{\frac{1}{q}}}{\|K_{a_{k}}^{\omega}\|_{L_{a}^{\frac{p}{q}}(\omega)}^{\frac{1}{q}}}\chi_{\{p:0<p\leq1\}}
+\frac{K_{a_{k}}^{\omega}(z)}{\|K_{a_{k}}^{\omega}\|_{L_{a}^{2-\frac{1}{p}}(\omega)}^{2-\frac{1}{p}}}\chi_{\{p:1<p\leq2\}},\quad z\in\mathbb{D}.
\]
For $t\in(0,1)$, $k_{4}\in\mathbb{N}$ and $\{y_{k}\}_{k=1}^{\infty}\in T_{p}^{p}(\{a_{k}\})$, choose test functions
\[
F_{t,p,k_{4},\{y_{k}\}}(z)=\frac{\sum\limits_{k=k_{4}}^{\infty}(1-|a_{k}|^{2})^{\frac{1}{p}}y_{k}r_{k}(t)f_{k,p}(z)}{\|\{y_{k}\}\|_{T_{p}^{p}(\{a_{k}\})}},\quad z\in\mathbb{D}.
\]
Combining Lemma \ref{regular-lemma}, \eqref{reproducing-formula}, \eqref{eq78}, \eqref{h-p}, \eqref{b-t-eq4} with \eqref{b-t-eq5}, for any
$z\in \overline{D(0,\varrho)}$ with $\varrho\in(0,1)$, for any $k_{4}>\max\{k_{1},k_{2}\}$, we deduce
\begin{align*}
\big|F_{t,p,k_{4},\{y_{k}\}}(z)\big|\leq&
\frac{1}{\|\{y_{k}\}\|_{T_{p}^{p}(\{a_{k}\})}}
\bigg(\sum_{k=k_{4}}^{\infty}(1-|a_{k}|^{2})|y_{k}|^{p}
\frac{|K_{a_{k}}^{\omega}(z)|^{\frac{p}{q}}}{\|K_{a_{k}}^{\omega}\|_{L_{a}^{\frac{p}{q}}(\omega)}^{\frac{p}{q}}}\bigg)^{\frac{1}{p}}\nonumber\\
\lesssim& \frac{\sup\limits_{k\geq k_{4}}|K_{a_{k}}^{\omega}(z)|^{\frac{1}{q}}}{\|\{y_{k}\}\|_{T_{p}^{p}(\{a_{k}\})}}
\bigg[\sum_{k=k_{4}}^{\infty}(1-|a_{k}|^{2})|y_{k}|^{p}\bigg(\widehat{\omega}(a_{k})(1-|a_{k}|)\bigg)^{\frac{p}{q}-1}\bigg]^{\frac{1}{p}}\nonumber\\
\leq & \sup\limits_{k\geq k_{4}}|K_{a_{k}}^{\omega}(z)|^{\frac{1}{q}}
\sup\limits_{k\geq k_{4}}\big(\widehat{\omega}(a_{k})(1-|a_{k}|)\big)^{\frac{p-q}{pq}}
\frac{\bigg(\sum\limits_{k=k_{4}}^{\infty}(1-|a_{k}|^{2})|y_{k}|^{p}\bigg)^{\frac{1}{p}}}{\|\{y_{k}\}\|_{T_{p}^{p}(\{a_{k}\})}}\nonumber\\
\lesssim&\varepsilon^{\frac{p-q}{pq}}\sup\limits_{k\geq k_{4}}|K_{a_{k}}^{\omega}(z)|^{\frac{1}{q}}, \quad p\in(0,1]
\end{align*}
and
\begin{align*}
\big|F_{t,p,k_{4},\{y_{k}\}}(z)\big|\leq&
\frac{1}{\|\{y_{k}\}\|_{T_{p}^{p}(\{a_{k}\})}}
\bigg(\sum_{k=k_{4}}^{\infty}(1-|a_{k}|^{2})|y_{k}|^{p}\bigg)^{\frac{1}{p}}
\bigg[\sum_{k=k_{4}}^{\infty}\bigg(\frac{|K_{a_{k}}^{\omega}(z)|}
{\|K_{a_{k}}^{\omega}\|_{L_{a}^{2-\frac{1}{p}}(\omega)}^{^{2-\frac{1}{p}}}}\bigg)^{\frac{p}{p-1}}\bigg]^{\frac{p-1}{p}}\nonumber\\
\lesssim& \sup\limits_{k\geq k_{4}}|K_{a_{k}}^{\omega}(z)|
\bigg(\sum_{k=k_{4}}^{\infty}\widehat{\omega}(a_{k})(1-|a_{k}|)\bigg)^{\frac{p-1}{p}}\nonumber\\
\asymp&\sup\limits_{k\geq k_{4}}|K_{a_{k}}^{\omega}(z)|
\bigg(\sum_{k=k_{4}}^{\infty}\kappa(D(a_{k},\tilde{\delta}))\bigg)^{\frac{p-1}{p}}\nonumber\\
\leq&\varepsilon^{\frac{p-1}{p}}\sup\limits_{k\geq k_{4}}|K_{a_{k}}^{\omega}(z)|, \quad p\in(1,2],
\end{align*}
which yields that for $p\in(0,2]$, $F_{t,p,k_{4},\{y_{k}\}}$ uniformly converges to 0 on any compact subset of $\mathbb{D}$ as $k_{4}\to\infty$.
Furthermore, by using Proposition \ref{le5}, \eqref{eq78} and \eqref{h-p}, for any $k_{4}\in\mathbb{N}$, we obtain
\[
\|F_{t,p,k_{4},\{y_{k}\}}\|_{L_{a}^{p}(\omega)}^{p}
\leq\frac{1}{\|\{y_{k}\}\|_{T_{p}^{p}(\{a_{k}\})}^{p}}
\int_{\mathbb{D}}\sum_{k=k_{4}}^{\infty}(1-|a_{k}|^{2})|y_{k}|^{p}\frac{|K_{a_{k}}^{\omega}(z)|^{\frac{p}{q}}}
{\|K_{a_{k}}^{\omega}\|_{L_{a}^{\frac{p}{q}}(\omega)}^{\frac{p}{q}}}\omega(z)dA(z)
\lesssim1,\quad p\in(0,1]
\]
and
\[
\|F_{t,p,k_{4},\{y_{k}\}}\|_{L_{a}^{p}(\omega)}^{p}
\lesssim\frac{\sum\limits_{k=k_{4}}^{\infty}(1-|a_{k}|^{2})|y_{k}|^{p}}{\|\{y_{k}\}\|_{T_{p}^{p}(\{a_{k}\})}^{p}}
\lesssim1,\quad p\in(1,2].
\]
Thus, we know that for $p\in(0,2]$, $\{F_{t,p,k_{4},\{y_{k}\}}\}$ is a sequence in $L_{a}^{p}(\omega)$ with
\[
\sup\limits_{k_{4}\in\mathbb{N}}\|F_{t,p,k_{4},\{y_{k}\}}\|_{L_{a}^{p}(\omega)}\lesssim1.
\]
It follows from the compactness of $J_{g}$ that
\[
\lim\limits_{k_{4}\to \infty} \|J_{g}(F_{t,p,k_{4},\{y_{k}\}})\|_{H^{q}}^{q}=0.
\]
Then there exists $k_{5}\in\mathbb{N}$ such that for any $k_{4}\in[k_{5},\infty)$, it holds that
\begin{equation}\label{b-t-eq6}
\|J_{g}(F_{t,p,k_{4},\{y_{k}\}})\|_{H^{q}}^{q}<\varepsilon.
\end{equation}
Note that $k_{5}$ does not depend on the choice of $\{y_{k}\}_{k=1}^{\infty}$ in $T_{p}^{p}(\{a_{k}\})$.
Combine with Lemma \ref{le1}, we obtain
\[
\int_{\mathbb{T}}\bigg(\int_{\Gamma(\xi)}|g'(z)|^{2}
\bigg|\sum_{k=k_{4}}^{\infty}(1-|a_{k}|^{2})^{\frac{1}{p}}y_{k}r_{k}(t)f_{k,p}(z)\bigg|^{2}dA(z)\bigg)^{\frac{q}{2}}d\xi
\asymp\|J_{g}(F_{t,p,k_{4},\{y_{k}\}})\|_{H^{q}}^{q}\|\{y_{k}\}\|_{T_{p}^{p}(\{a_{k}\})}^{q}.
\]
Integrate both sides of the inequality above from 0 to 1 with respect to $t$, by Fubini's theorem and \eqref{b-t-eq6}, for any $k_{4}\in[k_{5},\infty)$, we have
\begin{equation}\label{eq32}
\int_{\mathbb{T}}\int_{0}^{1}\bigg(\int_{\Gamma(\xi)}|g'(z)|^{2}
\bigg|\sum_{k=k_{4}}^{\infty}(1-|a_{k}|^{2})^{\frac{1}{p}}y_{k}r_{k}(t)f_{k,p}(z)\bigg|^{2}dA(z)\bigg)^{\frac{q}{2}}dtd\xi
\lesssim\varepsilon\|\{y_{k}\}\|_{T_{p}^{p}(\{a_{k}\})}^{q}.
\end{equation}
Notice that there exists $R_{4}\in(0,1)$ such that for all $R_{5}\in[R_{4},1)$, we have
\[
k\geq k_{5},\quad \textmd{for~any}~a_{k}\in\mathbb{D}\backslash R_{5}\mathbb{D}.
\]
Take a sequence $\{b_{k}\}_{k=1}^{\infty}$, where
\[
b_{k}=\frac{(1-|a_{k}|^{2})^{q}}{\widehat{\omega}(a_{k})^{\frac{q}{p}}}
\bigg(\sup\limits_{u\in D(a_{k},2\tilde{\delta})}|g'(u)|\bigg)^{q}\chi_{\{a_{k}:R_{4}\leq|a_{k}|<1\}}.
\]
Choose $s>1$. Let $\iota=\frac{2ps}{2ps-pq+2q}$. Then $\iota\leq1$.
Due to Lemma \ref{le2}, the dual space of $$T_{\iota}^{\big(\frac{ps}{p-q}\big)'}(\{a_{k}\})$$ can be identified with $$T_{\infty}^{\frac{ps}{p-q}}(\{a_{k}\}).$$
Take $$\{e_{k}\}_{k=1}^{\infty}\in T_{\iota}^{\big(\frac{ps}{p-q}\big)'}(\{a_{k}\}).$$
The factorization in Lemma \ref{le4} shows that there exist sequences $$\{x_{k}\}_{k=1}^{\infty}\in T_{\frac{2s}{2s-q}}^{s'}(\{a_{k}\}) \text{ and } \{y_{k}\}_{k=1}^{\infty}\in T_{p}^{p}(\{a_{k}\})$$ such that
$
e_{k}=x_{k}y_{k}^{\frac{q}{s}}.
$
Using a similar estimation as in \eqref{eq13}, we get
\begin{equation}\label{eq76}
|\langle e_{k},b_{k}^{\frac{1}{s}}\rangle_{T_{2}^{2}(\{a_{k}\})}|
\lesssim \|\{x_{k}\}\|_{T_{\frac{2s}{2s-q}}^{s'}(\{a_{k}\})}
\bigg[\int_{\mathbb{T}}\bigg(\sum\limits_{a_{k}\in\Gamma(\xi)\backslash R_{4}\mathbb{D}}
|y_{k}^{2}b_{k}^{\frac{2}{q}}|\bigg)^{\frac{q}{2}}d\xi\bigg]^{\frac{1}{s}}.
\end{equation}
Take $\tau>\max\{1,\frac{2}{q}\}$. By the subharmonicity of $|g'^{2}|$, \cite[Lemma 8]{PRS}, Lemma \ref{D-regular} and \eqref{reproducing-formula}, we deduce for $\delta$ small enough,
\begin{align*}
\sum\limits_{a_{k}\in\Gamma(\xi)\backslash R_{4}\mathbb{D}}|y_{k}^{2}b_{k}^{\frac{2}{q}}|
\asymp&\sum\limits_{a_{k}\in\Gamma(\xi)\backslash R_{4}\mathbb{D}}|y_{k}|^{2}
\bigg(\sup\limits_{u\in D(a_{k},2\tilde{\delta})}|g'(u)|\bigg)^{2}\bigg(\frac{1-|a_{k}|^{2}}{|1-\bar{a_{k}}\xi|}\bigg)^{\tau}\frac{(1-|a_{k}|^{2})^{2}}
{\widehat{\omega}(a_{k})^{\frac{2}{p}}}\nonumber\\
\lesssim&\sum\limits_{k=k_{5}}^{\infty}|y_{k}|^{2}
\frac{1}{\widehat{\omega}(a_{k})(1-|a_{k}|)}\int_{D(a_{k},3\tilde{\delta})}|g'(z)|^{2}\frac{\widehat{\omega}(z)}{1-|z|}
\bigg(\frac{1-|z|^{2}}{|1-\bar{z}\xi|}\bigg)^{\tau}
\frac{(1-|z|^{2})^{2}}{\widehat{\omega}(z)^{\frac{2}{p}}}dA(z)\nonumber\\
\asymp&\sum\limits_{k=k_{5}}^{\infty}|y_{k}|^{2}(1-|a_{k}|^{2})^{\frac{2}{p}}
\int_{D(a_{k},3\tilde{\delta})}|g'(z)|^{2}\bigg(\frac{1-|z|^{2}}{|1-\bar{z}\xi|}\bigg)^{\tau}|f_{k,p}(z)|^{2}dA(z)\nonumber\\
\leq&\int_{\mathbb{D}}\bigg(\frac{1-|z|^{2}}{|1-\bar{z}\xi|}\bigg)^{\tau}
\sum\limits_{k=k_{5}}^{\infty}|y_{k}|^{2}(1-|a_{k}|^{2})^{\frac{2}{p}}|g'(z)|^{2}|f_{k,p}(z)|^{2}dA(z).
\end{align*}
Together with Lemma \ref{le3}, we know
\begin{align}\label{r-8}
\int_{\mathbb{T}}\bigg(\sum\limits_{a_{k}\in\Gamma(\xi)\backslash R_{4}\mathbb{D}}|y_{k}^{2}b_{k}^{\frac{2}{q}}|\bigg)^{\frac{q}{2}}d\xi
&\lesssim\int_{\mathbb{T}}\bigg(\int_{\mathbb{D}}\bigg(\frac{1-|z|^{2}}{|1-\bar{z}\xi|}\bigg)^{\tau}
\sum\limits_{k=k_{5}}^{\infty}|y_{k}|^{2}(1-|a_{k}|^{2})^{\frac{2}{p}}|g'(z)|^{2}|f_{k,p}(z)|^{2}dA(z)\bigg)^{\frac{q}{2}}d\xi\nonumber\\
&\asymp\int_{\mathbb{T}}\bigg(\int_{\Gamma(\xi)}\sum\limits_{k=k_{5}}^{\infty}|y_{k}|^{2}(1-|a_{k}|^{2})
^{\frac{2}{p}}|f_{k,p}(z)|^{2}|g'(z)|^{2}dA(z)\bigg)^{\frac{q}{2}}d\xi.
\end{align}
Moreover, according to Lemma \ref{le-Khinchine}, Lemma \ref{le-Kahane} and Fubini's theorem, we obtain
\begin{eqnarray*}\label{eq9}\begin{split}
&~~~~\int_{\mathbb{T}}\bigg(\int_{\Gamma(\xi)}\sum\limits_{k=k_{5}}^{\infty}|y_{k}|^{2}(1-|a_{k}|^{2})
^{\frac{2}{p}}|f_{k,p}(z)|^{2}|g'(z)|^{2}dA(z)\bigg)^{\frac{q}{2}}d\xi\\&
\asymp\int_{\mathbb{T}}\bigg(\int_{\Gamma(\xi)}|g'(z)|^{2}\int_{0}^{1}\bigg|\sum\limits_{k=k_{5}}^{\infty}y_{k}r_{k}(t)
(1-|a_{k}|^{2})^{\frac{1}{p}}f_{k,p}(z)\bigg|^{2}dtdA(z)\bigg)^{\frac{q}{2}}d\xi\\&
=\int_{\mathbb{T}}\bigg(\int_{0}^{1}\int_{\Gamma(\xi)}\bigg|g'(z)\sum\limits_{k=k_{5}}^{\infty}y_{k}r_{k}(t)
(1-|a_{k}|^{2})^{\frac{1}{p}}f_{k,p}(z)\bigg|^{2}dA(z)dt\bigg)^{\frac{q}{2}}d\xi\\&
\asymp\int_{\mathbb{T}}\int_{0}^{1}\bigg(\int_{\Gamma(\xi)}\bigg|g'(z)\sum\limits_{k=k_{5}}^{\infty}y_{k}r_{k}(t)
(1-|a_{k}|^{2})^{\frac{1}{p}}f_{k,p}(z)\bigg|^{2}dA(z)\bigg)^{\frac{q}{2}}dtd\xi.
\end{split}\end{eqnarray*}
Together with \eqref{eq32} and \eqref{r-8}, we get
\begin{equation}\label{eq28}
\int_{\mathbb{T}}\bigg(\sum\limits_{a_{k}\in\Gamma(\xi)\backslash R_{4}\mathbb{D}}|y_{k}^{2}b_{k}^{\frac{2}{q}}|\bigg)^{\frac{q}{2}}d\xi
\lesssim\varepsilon\|\{y_{k}\}\|_{T_{p}^{p}(\{a_{k}\})}^{q}.
\end{equation}
According to \eqref{eq76} and \eqref{eq28}, it holds that
\[
|\langle e_{k},b_{k}^{\frac{1}{s}}\rangle_{T_{2}^{2}(\{a_{k}\})}|
\lesssim\|\{x_{k}\}\|_{T_{\frac{2s}{2s-q}}^{s'}(\{a_{k}\})}
\varepsilon^{\frac{1}{s}}\|\{y_{k}\}\|_{T_{p}^{p}(\{a_{k}\})}^{\frac{q}{s}}.
\]
By taking infimum over all possible factorizations of $\{e_{k}\}$, we deduce
\[
|\langle e_{k},b_{k}^{\frac{1}{s}}\rangle_{T_{2}^{2}(\{a_{k}\})}|\lesssim\
\varepsilon^{\frac{1}{s}}\|\{e_{k}\}\|_{T_{\iota}^{\big(\frac{ps}{p-q}\big)'}(\{a_{k}\})}.
\]
Therefore, using Lemma \ref{le2}, we know that for any $\rho_{1}\geq R_{4}$,
\begin{equation}\label{b-t-eq17}
\int_{\mathbb{T}}\bigg(\sup\limits_{a_{k}\in\Gamma_{\lambda}(\xi)\backslash\rho_{1}\mathbb{D}}\frac{(1-|a_{k}|^{2})}
{\widehat{\omega}(a_{k})^{\frac{1}{p}}}\bigg(\sup\limits_{u\in D(a_{k},2\tilde{\delta})}|g'(u)|\bigg)\bigg)^{\frac{pq}{p-q}}d\xi\lesssim \varepsilon^{\frac{p}{p-q}}.
\end{equation}
Hence, \eqref{eq27} holds.

Recall in Step 1, we can choose $\rho_{1}\in(0,\rho_{0}]$ such that $\rho_{1}$ tends to 1 as $\rho_{0}$ tends to 1.
Then we can choose $\rho_{0}\in(0,1)$ such that $\rho_{1}\geq R_{4}$.
Combining \eqref{b-t-eq14} with \eqref{b-t-eq17}, we deduce
\[
\int_{\mathbb{T}}\bigg(\sup\limits_{z\in\Gamma(\xi)\backslash\rho_{0}\mathbb{D}}\frac{(1-|z|^{2})}
{\widehat{\omega}(z)^{\frac{1}{p}}}|g'(z)|\bigg)^{\frac{pq}{p-q}}d\xi
\lesssim \varepsilon^{\frac{p}{p-q}}.
\]
Thus, we complete the whole proof of the necessity part.
\end{proof}
Combining Proposition \ref{t-c-1}, Proposition \ref{t-c-2}, Proposition \ref{t-c-3} with Proposition \ref{t4}, we get Theorem \ref{c-t1}. Moreover, one can get Theorem \ref{b-t1} by modifying the corresponding proofs of the above four propositions slightly.
\section{From Hardy space into Bergman space}
In this section, we are devoted to proving Theorem \ref{b-t2}.
To begin with, we state a result (cf. \cite[Theorem 5]{PR19}) which is vital to the proof of our main result in this section. It describes that the class of $\mathcal{D}$ weights is the largest class of radial weights such that the Littlewood-Paley estimate holds.
\begin{lemma}\label{r1} Let $0<p<\infty$ and $\omega$ be a radial weight. Then
\[
\|f\|_{L_{a}^{p}(\omega)}^{p}\asymp \int_{\mathbb{D}}|f'(z)|^{p}(1-|z|^{2})^{p}\omega(z)dA(z)
+|f(0)|^{p},\quad f\in H(\mathbb{D})
\]
holds if and only if $\omega\in\mathcal{D}$.
\end{lemma}
\begin{proof}[Proof of Theorem \ref{b-t2}] Let
\[
d\mu_{g}(z)=|g'(z)|^{q}(1-|z|^{2})^{q}\omega(z)dA(z).
\]
It follows from Lemma \ref{r1} that
\begin{equation}\label{eq52}
\|J_{g}(f)\|_{L_{a}^{q}(\omega)}^{q}\asymp\int_{\mathbb{D}}|f(z)|^{q}d\mu_{g}(z).
\end{equation}
By \eqref{eq52}, the boundedness (compactness) of $J_{g}:H^{p}\to L^{q}_{a}(\omega)$ is equivalent to that of $Id: H^{p}\to L^{q}(\mu_{g})$. For the case when $0<p\leq q<\infty$, the identity operator $Id: H^{p}\to L^{q}(\mu_{g})$ is bounded (compact) if and only if $\mu_{g}$ is a $\frac{q}{p}$-Carleson measure (vanishing $\frac{q}{p}$-Carleson measure) (cf. \cite[Section 9.2]{zhu2007}).

Now we are in a position to deal with the case when $0<q<p<\infty$. For $\xi\in\mathbb{T}$, recall
\[
\widetilde{\mu_{g}}(\xi)=\int_{\Gamma(\xi)}\frac{d\mu_{g}(\lambda)}{1-|\lambda|^{2}}.
\]
By \cite[Theorem E]{Pau-2016} and \eqref{eq52}, we know that $J_{g}:H^{p}\to L_{a}^{q}(\omega)$ is bounded if and only if $\widetilde{\mu_{g}}\in L^{\frac{p}{p-q}}(\mathbb{T})$.

It suffices to prove that $J_{g}:H^{p}\to L_{a}^{q}(\omega)$ is compact if $\widetilde{\mu_{g}}\in L^{\frac{p}{p-q}}(\mathbb{T})$.
Let $\{f_{j}\}_{j=1}^{\infty}$ be any sequence in $H^{p}$ with $\|f_{j}\|_{H^{p}}\leq 1$ and $f_{j}$ converges to 0 uniformly on any compact subset of $\mathbb{D}$.
By Lemma \ref{r1}, we obtain
\begin{equation}\label{eq96}
\|J_{g}(f_{j})\|_{L_{a}^{q}(\omega)}^{q}\asymp\int_{\mathbb{D}}|f_{j}(z)|^{q}d\mu_{g}(z).
\end{equation}
Then we only need to show
\begin{equation}\label{eq94}
\lim\limits_{j\to\infty}\int_{\mathbb{D}}|f_{j}(z)|^{q}d\mu_{g}(z)=0.
\end{equation}
Since $\widetilde{\mu_{g}}\in L^{\frac{p}{p-q}}(\mathbb{T})$, then for any $\varepsilon>0$, there exists $r_{0}\in(0,1)$ such that
\begin{equation}\label{eq95}
\bigg[\int_{\mathbb{T}}\bigg(\int_{\Gamma(\xi)\backslash r_{0}\mathbb{D}}
\frac{d\mu_{g}(z)}{1-|z|^{2}}\bigg)^{\frac{p}{p-q}}d\xi\bigg]^{\frac{p-q}{p}}<\varepsilon.
\end{equation}
Moreover, there exists $k_{0}\in\mathbb{N}$ such that for any $j>k_{0}$, $|f_{j}|^{q}<\varepsilon$ on $r_{0}\mathbb{D}$. It follows from
\cite[Theorem 3.1 in Chapter \uppercase\expandafter{\romannumeral2}]{Ga}, $\textmd{H}\ddot{\textmd{o}}\textmd{lder's}$ inequality, Fubini's theorem and \eqref{eq2},  that
\begin{align*}
\int_{\mathbb{D}}|f_{j}(z)|^{q}d\mu_{g}(z)
&\lesssim\varepsilon+\int_{\mathbb{D}\backslash r_{0}\mathbb{D}}\frac{|f_{j}(z)|^{q}}{1-|z|^{2}}\int_{I(z)}d\xi d\mu_{g}(z)\nonumber\\
&\leq\varepsilon+\bigg(\int_{\mathbb{T}}\sup_{z\in \Gamma(\xi)}|f_{j}(z)|^{p}d\xi\bigg)^{\frac{q}{p}}
\bigg[\int_{\mathbb{T}}\bigg(\int_{\Gamma(\xi)\backslash r_{0}\mathbb{D}}
\frac{d\mu_{g}(z)}{1-|z|^{2}}\bigg)^{\frac{p}{p-q}}d\xi\bigg]^{\frac{p-q}{p}}\nonumber\\
&\lesssim\varepsilon+\varepsilon\|f_{j}\|_{H^{p}}^{q}\lesssim\varepsilon.
\end{align*}
Thus, we get the desired conclusion and complete the proof.
\end{proof}

\section*{Acknowledgements}
This work is partially supported by National Natural Science Foundation of China.

\end{document}